\documentclass[12pt]{amsart}

\usepackage{amsfonts}
\usepackage{amsfonts,latexsym,rawfonts,amsmath,amssymb,amsthm}
\usepackage[plainpages=false]{hyperref}
\usepackage{graphicx}

\RequirePackage{color}
\theoremstyle{plain}

\newtheorem{corollary}{Corollary}

\newtheorem{lemma}{Lemma}

\newtheorem{remark}{Remark}

\newtheorem{theorem}{Theorem}
\numberwithin{equation}{section}

\newcommand{\impl}{\Rightarrow}
\newcommand{\beq}{\begin{equation}}
\newcommand{\eeq}{\end{equation}}
\newcommand{\beqs}{\begin{eqnarray*}}
\newcommand{\eeqs}{\end{eqnarray*}}
\newcommand{\beqn}{\begin{eqnarray}}
\newcommand{\eeqn}{\end{eqnarray}}
\newcommand{\beqa}{\begin{array}}
\newcommand{\eeqa}{\end{array}}
\newcommand{\deb}{\rightharpoonup}

\def\ve{\varepsilon}
\def\eps{\epsilon}
\def\eps{\varepsilon}

\def\R{\Bbb R}

\begin{document}

\title{Regularity in Monge's mass transfer problem}

\author{Qi-Rui Li}
\address{Centre for Mathematics and Its Applications, Australian National
University, Canberra, ACT 0200, Australia.}
\email{qi-rui.li@anu.edu.au}

\author{Filippo Santambrogio}
\address{Laboratoire de Mathematiques d'Orsay,
Universit\'e Paris-Sud,
91405 Orsay cedex,
France}
\email{filippo.santambrogio@math.u-psud.fr}

\author{Xu-Jia Wang}
\address{Centre for Mathematics and Its Applications, Australian National
University, Canberra, ACT 0200, Australia.}
\email{xu-jia.wang@anu.edu.au}



\thanks{ The first author was also supported by the Chinese Scholarship
Council, the second by the French ANR Project ISOTACE : ANR-12-MONU-0013, the third by the Australian Research Council.}

\begin{abstract}
In this paper, we study the regularity of optimal mappings in Monge's mass
transfer problem. Using the approximation to Monge's cost function $c(x,y)=|x-y|$ 
through the costs $c_{\varepsilon }(x,y)=\sqrt{\varepsilon^{2}+|x-y|^{2}}$, we consider the optimal mappings $T_\eps$ for these costs, and
we prove that the eigenvalues of the Jacobian matrix $DT_\eps$, which are all positive,
are locally uniformly bounded. 
By an example we prove that $T_\eps$ is in general
not uniformly Lipschitz continuous as $\eps\to 0$, 
even if the mass distributions are positive and smooth, and the domains are $c$-convex.
\end{abstract}

\maketitle

\baselineskip16.2pt

\parskip3pt

\section{Introduction}

The Monge mass transfer problem consists in finding an optimal mapping from one mass
distribution to another one such that the total cost is minimized among all
measure preserving mappings. This problem was first proposed by Monge \cite
{Monge} and has been studied by many authors in the last two hundred years: among the main achievements in the 20th century we cite \cite{Kant} and \cite{EG}. 

In Monge's problem, the cost of moving a mass from point $x$ to
point $y$ is proportional to the distance $|x-y|$, namely the cost function
is given by 
\begin{equation}\label{original cost}
c_0 (x,y)=|x-y|.
\end{equation}
This is a natural cost function. In the last two decades, due to a range of
applications, the optimal transportation for more general cost functions has
been a subject of extensive studies. In order to present the framework more precisely, let $\Omega $ and $\Omega ^{\ast }$ be
two bounded domains in the Euclidean space $\mathbb{R}^{n}$, and let $f$ and 
$g$ be two densities in $\Omega $ and $\Omega ^{\ast }$ respectively,
satisfying the mass balance condition 
\begin{equation}
\int_{\Omega }f(x)dx=\int_{\Omega ^{\ast }}g(y)dy.  \label{sec2 mass balance}
\end{equation}
Let $c$ be a smooth cost function defined on $\Omega \times \Omega ^{\ast }$.

The problem consists in finding a map $T:\Omega\to\Omega^*$ which solves
$$\min \int c(x,T(x))f(x)dx\quad:\quad T_\#f=g,$$
where the last condition reads ``the image measure of $f$ through $T$ is $g$'' and means $\int_A g(y)dy=\int_{T^{-1}(A)}f(x)dx$ for all subsets $A\subset \Omega^*$.

The existence and uniqueness of optimal mappings were obtained in \cite
{Bre, C96b, GM} if the cost function $c$ satisfies

\begin{itemize}
\item[(\textbf{A})] $\forall\ (x_0, y_0)\in \Omega\times \Omega^*$, the
mappings $x\in \overline \Omega\to D_y c(x, y_0)$ and $y\in \overline
\Omega^*\to D_x c(x_0,y)$ are diffeomorphisms onto their ranges.
\end{itemize}

The regularity of optimal mappings was a more complicated issue. 
For the quadratic cost function, it reduces to the regularity of the standard 
Monge-Ampere equation, of which the regularity has been studied by many 
authors (see for instance \cite {Ca92, Ca96a}). For general costs, the regularity was
obtained in \cite{MTW} if the domains satisfy a certain convexity condition, 
$f, g$ are positive and smooth, and the cost function $c$ satisfies the
following structure condition

\begin{itemize}
\item[(\textbf{B})] $\forall\ x\in \overline \Omega, y\in \overline \Omega^* 
$, and vectors $\xi, \eta\in \mathbb{R}^n$ with $\xi \perp \eta$, 
\begin{eqnarray*}
\sum_{i, j, k, l, p, q, r, s} \xi_i\xi_j\eta_k\eta_l \, (c_{ij,rs} -
c^{p,q}c_{ij,p}c_{q,rs})c^{r, k}c^{s, l}(x,y) \ge \beta_0|\xi|^2|\eta|^2,
\end{eqnarray*}
\end{itemize}

\noindent where $\beta _{0}$ is a positive constant. Loeper \cite{Loe}
showed that the optimal mapping may not be continuous if the condition 
({\bf B}) is violated, i.e. when there exist $\xi ,\eta \in 
\mathbb{R}^{n}$ with $\xi \perp \eta $ such that the left hand side is
negative. There are many follow-up researches on the regularity, 
in both the Euclidean space \cite {LTW, TW2} and on manifolds 
\cite {Am1, DG, FRV, KM, LV}. See also \cite {Vil} for recent development.

Monge's mass transfer problem is a prototype of the optimal transportation
and the function (\ref{original cost}) is the natural cost function.
Therefore the existence and regularity of optimal mappings for Monge's
problem are of particular interest. However this cost function does not
satisfy both key conditions, namely the condition (\textbf{A}) for the
existence and the condition (\textbf{B}) for the a priori estimates. 

The existence of optimal mappings for Monge's problem has been studied by
many researchers and was finally proved in \cite{CFM, TW}. Prior to that,
the existence was also obtained in \cite{EG} under some assumptions, 
and obtained in \cite{Sud}, with a gap fixed in \cite {Am}. 
See also \cite{AKP, CP0, CP} for the existence of optimal mappings 
when the norm \eqref{original cost}
is replaced by a more general norm in the Euclidean space. The proofs in 
\cite{CFM, TW} are very similar: both use the approximation $
|x-y|^{1+\varepsilon }\rightarrow |x-y|$ ($\varepsilon \rightarrow 0$). The key point is choosing an approximation with strictly convex costs of the difference $x-y$, which satisfy the assumption (\textbf{A}). The
optimal mapping for Monge's problem is not unique in general. But there is a
unique optimal mapping which is monotone on the transfer rays \cite{FM}.

In this paper we study the regularity of optimal mappings in Monge's mass
transfer problem. As the cost function (\ref{original cost}) does not
satisfy condition (\textbf{B}), the argument in \cite{MTW} does not apply to
Monge's problem. Indeed, Monge's problem also admits several minimizers $T$, even if a special one plays an important role: it is the transport map which is monotone on each transport ray (see  \cite{Am}: we will call this map monotone optimal trasnport). 

The regularity seems a rather tricky problem and very
little is known at the moment. Only in the 2 dimensional case, it was proved
in \cite{FGP} that the monotone optimal mapping is continuous in the
interior of the transfer set (i.e. the union of all transfer rays), under
the assumptions that the densities $f,g$ are positive, continuous, and with
compact, convex and disjoint supports. 

Our strategy to attack the regularity in Monge's problem is to establish
uniform estimates for the optimal mappings with respect to the cost function 
\begin{equation}
c_{\varepsilon }(x,y)=\sqrt{\varepsilon ^{2}+|x-y|^{2}}
\label{sec1 cost function}
\end{equation}
where $\varepsilon \in (0,1]$ is a constant. The cost function $
c_{\varepsilon }$ satisfies both conditions (\textbf{A}) and (\textbf{B}).
Therefore there is a unique optimal mapping $T_{\varepsilon }$ associated
with $c_{\varepsilon }$, given by 
\begin{equation*}
T_{\varepsilon }(x)=x-\frac{\varepsilon Du_{\varepsilon }}{\sqrt{1-|Du_{\varepsilon
}|^{2}}}.
\end{equation*}
where $u_{\varepsilon }$ is the potential function. By direct computation, $
u_{\varepsilon }$ satisfies the Monge-Amp\`{e}re equation \cite{MTW} 
\begin{equation}
\det w_{ij}=\frac{1}{\varepsilon ^{n}}\Big[1-|Du|^{2}\Big]^{\frac{n+2}{2}} 
\frac{f}{g},  \label{sec1 eq}
\end{equation}
with 
\begin{equation*}
\{w_{ij}\}=\Big\{\frac{\sqrt{1-|Du|^{2}}}{\varepsilon }\big(\delta
_{ij}-u_{x_{i}}u_{x_{j}}\big)-u_{x_{i}x_{j}}\Big\}.
\end{equation*}
Under appropriate assumptions, the a priori estimate 
\begin{equation}
\sup_{\Omega ^{\prime }}|D^{2}u_{\varepsilon }(x)|\leq C_{\varepsilon }\ \ \
\forall \ \Omega ^{\prime }\subset \subset \Omega .  \label{sec1 estimates}
\end{equation}
was established in \cite{MTW}, where the upper bound $C_{\varepsilon }$
depends on $\varepsilon $. Notice that the assumptions involve in particular lower bounds on the densities $f$ and $g$ on their respectives domains $\bar\Omega$ and $\bar\Omega^*$. these domains  should be $c_\ve-$convex w.r.t. each other, which typically reduces (if we want to impose it for all $\ve\to 0$) to the case of $\Omega\subset\Omega^*$, with $\Omega^*$ convex. In particular, this rules out the assumption of \cite{FGP}, since the supports will not be disjoint. The case we study is thus completely different form that of \cite{FGP}.

Equation (\ref{sec1 eq}) is strongly singular as $\varepsilon \rightarrow 0$
. Note that, due to the small $\varepsilon $, a uniform bound for $
D^{2}u_{\varepsilon }$ does not mean a uniform estimate for the optimal
mapping $T_{\varepsilon }$. Therefore we need to work directly on the
mapping $T_{\varepsilon }$. 

We wished to prove a uniform bound for $
DT_{\varepsilon }$, namely the uniform Lipschitz continuity of the optimal
mapping $T_{\varepsilon }$. By tedious computations, we are able to prove
that all the eigenvalues of the matrix $DT_{\varepsilon }$, which are all
positive, are locally uniformly bounded as $\varepsilon \rightarrow 0$.
This is one of the two main results of
the paper. Notice that this should bring some information on the behavior of $DT_{0}$, where $
T_{0}$ is the monotone
optimal mapping in Monge's problem. Yet, two problems arise: i) the condition on the eigenvalues being strongly nonlinear and applied to non-symmetric matrices, it is not easy to pass it to the limit, nor to give a meaning to the eigenvalues of $DT_0$ (which is a priori a distribution); ii) even the fact that the maps $T_\ve$ do converge to the monotonic optimal transport is not that easy if the supports of the measures are not disjoint (which is not the case for us).

However, as the matrix $DT_\varepsilon$ is - as we said - not symmetric, the boundedness of
the eigenvalues of $DT_\varepsilon$ does not imply the matrix itself is
uniformly bounded. Interestingly, we find that the matrix $DT_0$ is not
bounded in general. There exist positive and smooth $f, g$ such that $DT_0$
is unbounded at interior points (here by $T_0$ we mean the monotonic Monge optimal transport, and not the limit of $T_\ve$; however, it is possible to prove (see Section 4) that, should $DT_\ve$ be bounded, then $T_\ve\to T_0$, and hence this implies that $DT_{\varepsilon}$ cannot be 
uniformly bounded as $\varepsilon\to 0$). This is the second main result of
the paper.


This paper is arranged as follows. In section 2, we state our main estimate,
Theorem \ref{main estimate}. Section 3 is then devoted to the proof of
Theorem \ref{main estimate}. In section 4, we provide positive and smooth
densities $f,g$ such that the monotonic optimal mapping $T_{0}$ is not
Lipschitz continuous at interior points. We conclude the paper with some remarks and perspectives in Section 5.

\vskip30pt

\section{Uniform a priori estimates}

Let $c=c_{\varepsilon }$ be the cost function given in (\ref{sec1 cost
function}). The optimal mapping $T=T_{\varepsilon }:\Omega \rightarrow
\Omega ^{\ast }$ is given by \cite{MTW} 
\begin{equation}
T(x)=\left[ D_{x}c(x,\cdot )\right] ^{-1}Du(x),  \label{sec2 general T}
\end{equation}
where $u=u_{\varepsilon }$ is a $c$-concave potential function. In this and
the next sections, we deal with the a priori estimates for $DT$. We will
omit the subscript $\varepsilon $ when no confusions arise.

From \cite{MTW}, the potential function $u$ satisfies the fully nonlinear
PDE of Monge-Amp\`{e}re type, 
\begin{equation}
\det (D_{x}^{2}c-D^{2}u)=|\det D_{xy}^{2}c|\frac{f}{g\circ T}\text{ in }
\Omega .  \label{sec2 general pde}
\end{equation}
For the cost function (\ref{sec1 cost function}), one has 
\begin{equation}
Dc(x,y)=\frac{x-y}{\sqrt{\varepsilon ^{2}+|x-y|^{2}}}.  \label{sec2 temp1}
\end{equation}
Hence by (\ref{sec2 general T}), 
\begin{equation}
T(x)=x-L(x)Du(x),  \label{sec2 T}
\end{equation}
where the function $L$ is given by 
\begin{equation}
L(x)=:\frac{\varepsilon }{\sqrt{1-v}}  \label{sec2 L}
\end{equation}
and 
\begin{equation}
v=:|Du|^{2}.  \label{sec2 v}
\end{equation}
From (\ref{sec2 T}) and (\ref{sec2 L}), we can solve 
\begin{equation}
v=\frac{d^{2}\left( x\right) }{\varepsilon ^{2}+d^{2}\left( x\right) }
\label{sec2 v invariant under coord}
\end{equation}
and consequently 
\begin{equation}
L=\sqrt{\varepsilon ^{2}+d^{2}\left( x\right) },
\label{sec2 L invariant under coord}
\end{equation}
where 
\begin{equation}
d\left( x\right) =\left\vert x-T\left( x\right) \right\vert .
\end{equation}
As in \cite{MTW} , we denote 
\begin{eqnarray}
A_{ij}(x) &=&D_{x^{i}x^{j}}^{2}c(x,T(x))  \label{sec2 matrixA} \\
&=&\frac{1}{L}(\delta _{ij}-D_{i}uD_{j}u).  \notag
\end{eqnarray}
Then equation (\ref{sec2 general pde}) can be written in the form 
\begin{equation}
\det w_{ij}=\frac{\varepsilon ^{2}}{L^{n+2}}\frac{f}{g\circ T},
\label{sec2 pde}
\end{equation}
where 
\begin{equation}
w_{ij}=:A_{ij}-D_{ij}^{2}u  \label{sec2 matrixW}
\end{equation}
is a nonnegative symmetric matrix.

We observe from (\ref{sec2 matrixA}) that $A_{ij}$ is positive definite, and
the inverse matrix of $A_{ij}$ is given by 
\begin{equation}
A^{ij}=L\big(\delta _{ij}+\frac{L^{2}}{\varepsilon ^{2}}D_iuD_ju\big).
\label{sec2 inverse of A}
\end{equation}
Let us denote 
\begin{equation}
W=:\sum_{\alpha ,\beta =1}^{n}A^{\alpha \beta }w_{\alpha \beta }.
\label{sec2 W}
\end{equation}
Then we have following uniform estimates:

\begin{theorem}
\label{main estimate} Suppose $\Omega ,\Omega ^{\ast }$ are bounded domains
in $\mathbb{R}^{n}$ ($n\geq 2)$, $f\in C^{1,1}(\Omega ),g\in C^{1,1}(\Omega
^{\ast })$, $f,g$ have positive upper and lower bounds, and (\ref{sec2 mass
balance}) holds. Let $u\in C^{3,1}(\Omega )$ be a $c$-concave solution to ( 
\ref{sec2 pde}), then we have a priori estimate 
\begin{equation}
W(x)\leq C,  \label{boundedness W}
\end{equation}
where $C$ depends on $n$, $dist(x,\partial \Omega )$, $f$ and $g$, but is
independent of the constant $\varepsilon \in (0,1]$.
\end{theorem}

By (\ref{sec2 T}) (\ref{sec2 matrixW}) and (\ref{sec2 inverse of A}), it is
ready to check that the Jacobian matrix of $T$ is given by 
\begin{eqnarray}
T_{j}^{i} &=&\delta _{ij}-L_{j}u_{i}-Lu_{ij}  \label{sec2 DT} \\
&=&\delta _{ij}-L(u_{ij}+\frac{L^{2}}{\varepsilon ^{2}}u_{i}
\sum_{k}u_{k}u_{kj})  \notag \\
&=&\sum_{k}A^{ik}w_{kj}.  \notag
\end{eqnarray}
Since the matrices $\{A^{ij}\}$ and $\{w_{ij}\}$ are positive, then $DT$ is diagonalizable, and its
eigenvalues $\lambda _{1,}\cdots ,\lambda _{n}$ of Jacobian $DT$ are
positive, and $\sum_{i=1}^{n}\lambda _{i}=W$. So if $W$ is bounded, one
immediately sees that all the eigenvalues of $DT$ are bounded from above and
below. We therefore have

\begin{corollary}
Under the assumptions of Theorem \ref{main estimate}, we have for any $
\Omega^{\prime }\subset\subset\Omega$, 
\begin{equation}
C^{-1}\leq \min_{i}\lambda _{i}\leq \max_{i}\lambda _{i}\leq C\ \ \text{in}\
\Omega^{\prime },
\end{equation}
where $C$ depends on $n$, $dist(\Omega^{\prime },\partial \Omega )$, $f$ and 
$g$, but is independent of $\varepsilon \in (0,1]$.
\end{corollary}

In view of (\ref{sec2 inverse of A}) and (\ref{sec2 W}), one finds that 
\begin{equation*}
W=L\sum_{i}w_{ii}+\frac{L^{3}}{\varepsilon ^{2}}\sum_{i,j}u_{i}u_{j}w_{ij}.
\end{equation*}
Hence we obtain 
\begin{equation*}
L\sum w_{ii}\leq C
\end{equation*}
By \eqref{sec2 matrixA} and \eqref{sec2 matrixW}, we also obtain the
estimate for $D^2 u$.

\begin{corollary}
Under the assumptions of Theorem \ref{main estimate}, we have for any $
\Omega^{\prime }\subset\subset\Omega$, 
\begin{equation}
|D^{2}u|\leq C/L \ \ \text{in}\ \Omega^{\prime },
\end{equation}
where $C$ depends on $n$, $dist(\Omega^{\prime },\partial \Omega )$, $f$ and 
$g$, but is independent of $\varepsilon \in (0,1]$.
\end{corollary}

Corollary 2 morally gives a $C^{1,1}$ estimate for the potential function $u_{0}=\lim_{\varepsilon
\rightarrow 0}u$  on the set 
\begin{equation*}
E_{\delta }=\left\{ x\in \Omega :|T(x)-x|\geq \delta >0\right\} .
\end{equation*}
This recovers a well-known result which reads ``the potential is $C^{1,1}$ in the interior of transport rays'', which was also used by \cite{Am} in order to prove the countable Lipschitz property of the direction of $Du$. At a point $x_{0}$ with $v(x_{0})>0$, denote 
\begin{equation*}
\nu =-\frac{Du\left( x_{0}\right) }{\left\vert Du\left( x_{0}\right)
\right\vert },
\end{equation*}
and let $\xi ^{\alpha }$ be unit vectors such that $\left\{ \nu ,\xi
^{\alpha }\right\} _{\alpha =1,...,n-1}$ are orthonormal. We denote 
\begin{eqnarray*}
T_{\nu }^{\nu } &=&\sum_{i,j}\nu _{i}\nu _{j}T_{j}^{i}, \\
T_{\xi ^{\alpha }}^{\xi ^{\alpha }} &=&\sum_{i,j}\xi _{i}^{\alpha }\xi
_{j}^{\alpha }T_{j}^{i}.
\end{eqnarray*}
By (\ref{sec2 DT}) and (\ref{sec2 inverse of A}), 
\begin{eqnarray}
D_{\nu }\left\langle \nu ,T\right\rangle &=&T_{\nu }^{\nu }=\sum_{i,j,k}\nu
_{i}A^{ik}w_{kj}\nu _{j}  \label{sec2 rk1} \\
&=&L\sum_{i,j,k}\left( 1+\frac{vL^{2}}{\varepsilon ^{2}}\right) \nu
_{k}w_{kj}\nu _{j}  \notag \\
&=&\frac{L^{3}}{\varepsilon ^{2}}\sum_{j,k}\nu _{k}w_{kj}\nu _{j}.  \notag
\end{eqnarray}
Similarly, 
\begin{equation}
D_{\xi ^{\alpha }}\left\langle \xi ^{\alpha },T\right\rangle =T_{\xi
^{\alpha }}^{\xi ^{\alpha }}=L\sum_{j,k}\xi _{k}^{\alpha }w_{kj}\xi
_{j}^{\alpha }.  \label{sec2 rk2}
\end{equation}
Noticing that $\{w_{ij}\}$ is positive definite, it is clear from (\ref{sec2
rk1}) and (\ref{sec2 rk2}) that $T_{\nu }^{\nu }$ and $T_{\xi ^{\alpha
}}^{\xi ^{\alpha }}$ are positive. Recall that 
\begin{equation*}
W=T_{\nu }^{\nu }+\sum_{\alpha =1}^{n-1}T_{\xi ^{\alpha }}^{\xi ^{\alpha }}.
\end{equation*}
Hence by (\ref{boundedness W}) we obtain

\begin{corollary}
Under the assumptions of Theorem \ref{main estimate}, we have for any $
\Omega^{\prime }\subset\subset\Omega$, 
\begin{equation}
\left\{ {
\begin{split}
&T_{\nu }^{\nu }\leq C \\
&T_{\xi ^{\alpha }}^{\xi ^{\alpha }}\leq C
\end{split}
}\right.\ \ \ \ \text{in}\ \Omega',
\end{equation}
where $C$ depends on $n$, $dist(\Omega^{\prime },\partial \Omega )$, $f$ and 
$g$, but is independent of $\varepsilon \in (0,1]$.
\end{corollary}

At the limit, this corresponds to saying (even if what we state here is not rigorous) that the limit mapping $T_{0}$ is Lipschitz continuous in the direction of transfer
rays, and for any unit vector $\xi $ perpendicular to the transfer rays, $
\langle \xi ,T_{0}\rangle $ is Lipschitz continuous in the $\xi $-direction. The Lipschitz continuity along transport rays is not surprising, since we are doing one-dimensional optimal transport between two measures with upper and lower bounds on their densities; yet, the densities of the one-dimensional problem along each ray are affected by a Jacobian factor (due to the decomposition of $f$ and $g$ along rays), and this makes this Lipschitz result not completely evident.
In section 4, we will construct an example to show that the component $
\left\langle \nu _{0},T_{0}\right\rangle $ is in general not Lipschitz
continuous in $\xi $, even though the mass distributions are positive and
smooth, where $\nu _{0}$ is a direction of transfer rays and $\xi \bot \nu
_{0}$.

\vskip10pt

In Theorem \ref{main estimate}, we assume that $u\in C^{3,1}$. This
assumption is not needed if $\Omega \subset \Omega ^{\ast}$ and $
\Omega^{\ast }$ is a bounded convex domain in $\mathbb{R}^{n}$, as it
implies that $\Omega^{\ast }$ is $c^{\ast }$-convex with respect to $\Omega$
and by approximation, and the condition $u\in C^{3,1}$ is always satisfied, see \cite{MTW}.

\vskip30pt

\section{Proof of Theorem \protect\ref*{main estimate}}

To prove Theorem \ref{main estimate}, we introduce the auxiliary function 
\begin{equation}
H(x)=\eta (x)W(x),  \label{sec3 auxiliary function}
\end{equation}
where $\eta $ is a cut-off function. Suppose that $H$ attains its maximum at
some point $x_{0}$. To prove that $H(x_{0})$ is uniformly bounded in $
\varepsilon $, the computation is rather complicated. We find the
computation can be made a little simpler if we first make a linear
transformation such that 
\begin{equation}
A_{ij}(x_{0})=\delta _{ij},  \label{sec3 coord1}
\end{equation}
and then make a rotation of coordinates such that 
\begin{equation}
w_{ij}(x_{0})=diag\left\{ \lambda _{1},...,\lambda _{n}\right\} .
\label{sec3 coord2}
\end{equation}
It is well-known that $A_{ij},$ $w_{ij}$ are tensors \cite{KM}. An advantage
of working on tensors is that one may choose a particular coordinate system
to simplify the computation. As we only made a linear transform on the
Euclidean space $\mathbb{R}^{n}$, the Riemannian curvature tensor under the
metric $\sigma _{ij}$ vanishes, which allows us to exchange the derivatives
freely. In the following we will use $D$ to denote the normal derivatives in 
$\mathbb R^n$ and $\nabla $ to denote covariant derivatives under the
metric $\sigma$.

Suppose the linear transformation is given by $y=P^{-1} x$ (i.e. $x_{i}=\sum
P_{ik}y_{k}$) such that $P^{T}AP=I$ is the unit matrix at $x_0$. Then by ( 
\ref{sec2 matrixA}) and (\ref{sec2 matrixW}), 
\begin{eqnarray*}
\bar{A}_{ij} &=&\sum_{k,l}A_{kl}P_{ki}P_{lj}=\left( P^{T}AP\right) _{ij}, \\
\bar{w}_{ij} &=&\sum_{k,l}w_{kl}P_{ki}P_{lj}=\left( P^{T}wP\right) _{ij},
\end{eqnarray*}
where bar denotes quantities in the $y$-coordinates.

Denote $\left\{ \sigma _{ij}\right\} =P^{T}P$, and $\left\{ \sigma
^{ij}\right\} =\left( P^{T}P\right) ^{-1}$. Then by (\ref{sec2 matrixA}) and
(\ref{sec2 inverse of A}), 
\begin{eqnarray}
\delta _{ij} &=&\bar{A}_{ij}=\left( P^{T}AP\right) _{ij}  \label{sec3 temp0a}
\\
&=&\frac{1}{L}(\sigma _{ij}-\bar{u}_{i}\bar{u}_{j})  \notag
\end{eqnarray}
and 
\begin{eqnarray}
\delta _{ij} &=&\bar{A}^{ij}=\left( P^{-1}A^{-1}(P^{T})^{-1}\right) _{ij}
\label{sec3 temp0b} \\
&=&L\big(\sigma ^{ij}+\frac{L^{2}}{\varepsilon ^{2}}\bar{u}^{i}\bar{u}^{j}
\big),  \notag
\end{eqnarray}
where $\bar{u}^{i}=:\sum_{k}\sigma ^{ik}\bar{u}_{k}$. Note that by (\ref
{sec2 v invariant under coord}) and (\ref{sec2 L invariant under coord}), $v$
and $L$ are invariant under the coordinate transformation, and 
\begin{eqnarray}
&&\bar{v}=\sum \sigma ^{ij}\bar{u}_{i}\bar{u}_{j}=\sum \bar{u}^{i}\bar{u}
_{i}\leq 1,  \label{sec3 temp0c} \\
&&\hskip40pt\varepsilon ^{2}\leq \bar{L}^{2}\leq C.  \label{sec3 temp0d}
\end{eqnarray}

For simplicity we will omit the bar below. In view of (\ref{sec3 temp0a}),
we have, at $x_{0}$, 
\begin{equation}
\sigma _{ij}=L\delta _{ij}+\nabla _{i}u\nabla _{j}u.  \label{sec3 temp0.1}
\end{equation}
By (\ref{sec3 temp0b}), we have, at $x_{0}$, 
\begin{eqnarray*}
u_{i} &=&{\sum }_{j}\delta _{ij}u_{j} \\
&=&{\sum }_{j}L\big(\sigma ^{ij}+\frac{L^{2}}{\varepsilon ^{2}}u^{i}u^{j}
\big)u_{j} \\
&=&L\big(1+\frac{L^{2}}{\varepsilon ^{2}}v\big)u^{i},
\end{eqnarray*}
where $u_{i}=\nabla _{i}u$. By (\ref{sec2 L}), it follows that 
\begin{equation}
u_{i}=\frac{Lu^{i}}{1-v}=\frac{L^{3}}{\varepsilon ^{2}}u^{i}.
\label{sec3 temp0.2}
\end{equation}

Hence $u^{i}=\frac{\varepsilon ^{2}}{L^{3}}u_{i}$ and by \eqref{sec3 temp0b}
, 
\begin{equation}
\sigma ^{ij}=\frac{1}{L}\Big(\delta _{ij}-\frac{\varepsilon ^{2}}{L^{3}}
u_{i}u_{j}\Big).  \label{sec3 temp0.3}
\end{equation}
Formulas \eqref{sec3 temp0.1}, \eqref{sec3 temp0.2} and \eqref{sec3 temp0.3}
will be repeatedly used in our calculation below. Without loss of
generality, we may also assume 
\begin{equation}
\lambda _{1}\geq \lambda _{2}\geq \cdot \cdot \cdot \geq \lambda _{n}
\label{sec3 order for eigenvalues}
\end{equation}
at $x_{0}$.

Since $x_0$ is the maximum point, we have 
\begin{eqnarray}
0 &=&\nabla _{i}\log H(x_0)=\frac{\eta _{i}}{\eta }+\frac{W_{i}}{W}
\label{sec3 critical condition} \\
&=&\frac{\eta _{i}}{\eta }+\frac{\sum w_{\alpha \alpha ;i}}{W}-\frac{\sum
A_{\alpha \alpha ;i}w_{\alpha \alpha }}{W},  \notag \\
0 &\geq &\nabla _{ij}^{2}\log H(x_0)=\frac{\eta _{ij}}{\eta }-2\frac{\eta
_{i}\eta _{j}}{\eta ^{2}}+\frac{W_{ij}}{W}  \notag
\end{eqnarray}
as a matrix, where subscripts $i,j$ on the R.H.S. denote covariant
derivatives in the metric $\sigma $. We thus obtain, at $x_{0}$, 
\begin{equation}
0\geq W\sum w^{ij}(\frac{\eta _{ij}}{\eta }-2\frac{\eta _{i}\eta _{j}}{\eta
^{2}})+\sum w^{ij}W_{ij},  \label{sec3 temp1}
\end{equation}
where $w^{ij}$ is the inverse of $w_{ij}$.

Differentiating (\ref{sec2 pde}) gives 
\begin{eqnarray}
\sum w^{ij}w_{ij;a} &=&\varphi _{a},  \label{differentiating pde} \\
\sum w^{ij}w_{ij;ab} &=&\sum w^{is}w^{jt}w_{ij;a}w_{st;b}+\varphi _{ab},
\label{differentiating pde 2nd}
\end{eqnarray}
where $\varphi $ is given by 
\begin{equation}
\varphi =\log \Big(\frac{\varepsilon ^{2}}{L^{n+2}}\frac{f}{g\circ T}\Big).
\label{sec3 phi}
\end{equation}
In our computation we use the notation $w_{ij;k}=\nabla _{k}w_{ij}$, $
w_{ij;kl}=\nabla _{l}\nabla _{k}w_{ij}$, $A_{ij;k}=\nabla _{k}A_{ij}$ and $
A_{ij;kl}=\nabla _{l}\nabla _{k}A_{ij}$.

To estimate the term $\sum w^{ij}W_{ij}$ in (\ref{sec3 temp1}), we first
prove the following lemma.

\begin{lemma}
\label{DA} We have 
\begin{eqnarray*}
A_{ij;k} &=&\frac{L^{2}}{\varepsilon ^{2}}A_{ij}u^{h}w_{hk}+\frac{1}{L}
(u_{j}w_{ik}+u_{i}w_{jk}) \\
&& -\frac{1}{L}(A_{ij}u_{k}+A_{ik}u_{j}+u_{i}A_{jk}), \\
A_{ii;\beta }-A_{i\beta ;i} &=&\frac{L^{2}}{\varepsilon ^{2}}
(A_{ii}u^{t}w_{t\beta }-A_{i\beta }u^{t}w_{ti}) +\frac{1}{L}(w_{i\beta
}u_{i}-w_{ii}u_{\beta }),
\end{eqnarray*}
(we use the summation convention $u^{h}w_{hk}=\sum_h u^{h}w_{hk}$).
\end{lemma}

\begin{proof}
Recall that $v=\sigma ^{ij}u_{i}u_{j}=\sum u^{i}u_{i}$. Therefore 
\begin{eqnarray*}
\frac{dL}{dv} &=&\frac{1}{2}\frac{L^{3}}{\varepsilon ^{2}}, \\
v_{k} &=&2u^{h}u_{hk}.
\end{eqnarray*}
By \eqref{sec3 temp0a}, $A_{ij}=\frac{1}{L}(\sigma _{ij}- {u}_{i} {u}_{j})$.
Differentiating, we get 
\begin{eqnarray}
A_{ij;k} &=&-\frac{1}{L^{2}}\frac{dL}{dv}v_{k}(\sigma _{ij}-u_{i}u_{j})+ 
\frac{1}{L}(-u_{ik}u_{j}-u_{i}u_{jk})  \label{lem1 temp1} \\
&=&-\frac{L^{2}}{\varepsilon ^{2}}A_{ij}u^{h}u_{hk}+\frac{1}{L}
(-u_{ik}u_{j}-u_{i}u_{jk})  \notag \\
&=&\frac{L^{2}}{\varepsilon ^{2}}A_{ij}u^{h}w_{hk}+\frac{1}{L}
(w_{ik}u_{j}+u_{i}w_{jk})  \notag \\
&&-\frac{1}{L}(A_{ij}u_{k}+A_{ik}u_{j}+u_{i}A_{jk}).  \notag
\end{eqnarray}
The second formula follows from (\ref{lem1 temp1}) immediately.
\end{proof}

Differentiating (\ref{sec2 W}) twice and using $A_{ij}\left( x_{0}\right)
=\delta _{ij}$, 
\begin{eqnarray}
{\sum }_{i,j}w^{ij}W_{ij} &=&\sum w^{ij}w_{\alpha \alpha ;ij}-2\sum
w^{ij}A_{\alpha \beta ;i}w_{\alpha \beta ;j}-\sum w^{ij}A_{\alpha \alpha
;ij}w_{\alpha \alpha }  \label{sec3 temp2.1} \\
&&+2\sum w^{ij}A_{\alpha k;i}A_{\beta k;j}w_{\alpha \beta }  \notag \\
&\geq &\sum w^{ij}w_{\alpha \alpha ;ij}-2\sum w^{ij}A_{\alpha \beta
;i}w_{\alpha \beta ;j}-\sum w^{ij}A_{\alpha \alpha ;ij}w_{\alpha \alpha }. 
\notag
\end{eqnarray}
We have by (\ref{differentiating pde 2nd}) 
\begin{eqnarray}
{\sum }_{i,j,\alpha }w^{ij}w_{\alpha \alpha ;ij} &=&\sum w^{ij}A_{\alpha
\alpha ;ij}-\sum w^{ij}u_{\alpha \alpha ij}  \label{sec3 temp2.2} \\
&=&\sum w^{ij}w_{ij;\alpha \alpha }+\sum w^{ij}\left( A_{\alpha \alpha
;ij}-A_{ij;\alpha \alpha }\right)   \notag \\
&\geq &\sum \varphi _{\alpha \alpha }+\sum w^{ii}\left( A_{\alpha \alpha
;ii}-A_{ii;\alpha \alpha }\right) .  \notag
\end{eqnarray}
By the first formula in Lemma \ref{DA} 
\begin{eqnarray*}
{\sum }_{i,j,\alpha ,\beta }w^{ij}A_{\alpha \beta ;i}w_{\alpha \beta ;j} &=&
\frac{L^{2}}{\varepsilon ^{2}}\sum w_{\alpha \alpha ;i}u^{i}-\frac{1}{L}\sum
w^{ij}u_{j}w_{\alpha \alpha ;i} \\
&&+\frac{2}{L}\sum u_{\beta }w_{\alpha \beta ;\alpha }-\frac{2}{L}\sum
w^{ij}u_{\beta }w_{j\beta ;i}.
\end{eqnarray*}
By (\ref{sec3 temp0.2}), it follows 
\begin{eqnarray*}
{\sum }_{i,j,\alpha ,\beta }w^{ij}A_{\alpha \beta ;i}w_{\alpha \beta ;j} &=&3
\frac{L^{2}}{\varepsilon ^{2}}\sum w_{\alpha \alpha ;i}u^{i}-\frac{1}{L}\sum
w^{ij}u_{j}w_{\alpha \alpha ;i} \\
&&+\frac{2}{L}\sum u_{\beta }(A_{\alpha \beta ;\alpha }-A_{\alpha \alpha
;\beta })-\frac{2}{L}\sum w^{ij}u_{\beta }w_{ij;\beta } \\
&&-\frac{2}{L}\sum w^{ii}u_{\beta }(A_{i\beta ;i}-A_{ii;\beta }).
\end{eqnarray*}
By (\ref{differentiating pde}) and the second formula in Lemma \ref{DA}, we
then obtain
\begin{eqnarray*}
{\sum }_{i,j,\alpha ,\beta }w^{ij}A_{\alpha \beta ;j}w_{\alpha \beta ;i} &=&3
\frac{L^{2}}{\varepsilon ^{2}}\sum w_{\alpha \alpha ;i}u^{i}-\frac{1}{L}\sum
w^{ii}u_{i}w_{\alpha \alpha ;i} \\
&&-2\frac{L^{2}}{\varepsilon ^{2}}\sum \varphi _{\beta }u^{\beta }+2\frac{Lv
}{\varepsilon ^{2}}(W-n) \\
&&+2\frac{L}{\varepsilon ^{2}}(\mathcal{W}-n)\sum w_{ii}u_{i}u^{i},
\end{eqnarray*}
where 
\begin{equation*}
\mathcal{W}=:\sum w^{ii}=\sum \frac{1}{\lambda _{i}}.
\end{equation*}
Recalling \eqref{sec3 temp0.2} and \eqref{sec3 temp0c}, 
\begin{equation}
0\leq u^{i}u_{i}\leq \sum u^{i}u_{i}\leq 1  \label{new1}
\end{equation}
for any given $i$. Hence
\begin{eqnarray}
{\sum }_{i,j,\alpha ,\beta }w^{ij}A_{\alpha \beta ;j}w_{\alpha \beta ;i}
 &\leq & 3\frac{L^{2}}{\varepsilon ^{2}}\sum w_{\alpha \alpha ;i}u^{i}
         -\frac{1}{L}\sum w^{ii}u_{i}w_{\alpha \alpha ;i} \label{sec3 temp2.3} \\
 &&-2\frac{L^{2}}{\varepsilon ^{2}}\sum \varphi _{\beta }u^{\beta }+\frac{L}{
\varepsilon^{2}}Q.  \notag
\end{eqnarray}
Here and below we use $Q$ to denote quantities satisfying 
\begin{equation*}
Q\leq C\Big(1+\frac{W}{\eta }+W^{2}+\frac{1}{\eta }W\mathcal{W}\Big).
\end{equation*}
Inserting (\ref{sec3 temp2.2}) and (\ref{sec3 temp2.3}) into (\ref{sec3
temp2.1}), we obtain 
\begin{eqnarray}
\sum w^{ij}W_{ij} &\geq &\sum w^{ii}(A_{\alpha \alpha ;ii}-A_{ii;\alpha
\alpha })-\sum w^{ii}A_{\alpha \alpha ;ii}w_{\alpha \alpha }
\label{sec3 temp2.4} \\
&&-6\frac{L^{2}}{\varepsilon ^{2}}\sum w_{\alpha \alpha ;i}u^{i}+\frac{2}{L}
\sum w^{ii}u_{i}w_{\alpha \alpha ;i}  \notag \\
&&+\sum \varphi _{\alpha \alpha }+\frac{4L^{2}}{\varepsilon ^{2}}\sum
\varphi _{\alpha }u^{\alpha }-\frac{L}{\varepsilon ^{2}}Q.  \notag
\end{eqnarray}

To proceed further, we need the following lemma.

\begin{lemma}
\label{DDA} We have 
\begin{eqnarray*}
{\sum}_{i,\alpha} w^{ii}(A_{\alpha \alpha ;ii}-A_{ii;\alpha \alpha }) &\geq
& \frac{L}{\varepsilon ^{2}}\mathcal{W}\sum w_{ii}^{2} -\frac{2}{L}\sum
w^{ii}u_{i}w_{\alpha \alpha ;i} \\
&& - \frac{L^{2}}{\varepsilon ^{2}}\mathcal{W}\sum w_{\alpha \alpha; h}u^{h}
+ (n+2)\frac{L^{2}}{\varepsilon ^{2}}\sum \varphi _{\beta }u^{\beta } - 
\frac{L}{\varepsilon ^{2}}Q
\end{eqnarray*}
and 
\begin{eqnarray*}
{\sum}_{i,\alpha} w^{ii}A_{\alpha \alpha ;ii}w_{\alpha \alpha } &\leq & - 2 
\frac{L}{\varepsilon ^{2}}\mathcal{W}\sum w_{ii}^{2}u_{i}u^{i} - \frac{L}{
\varepsilon^{2}}\mathcal{W}W\sum w_{ii}u_{i}u^{i} \\
&& +\frac{L^{2}}{\varepsilon ^{2}}W\sum \varphi _{\beta }u^{\beta } +\frac{
2L^{2}}{\varepsilon ^{2}}\sum w_{ii}u^{i}\varphi _{i} +\frac{L}{
\varepsilon^{2}}Q.
\end{eqnarray*}
\end{lemma}

\begin{proof}
In view of Lemma \ref{DA}, 
\begin{equation}
A_{\alpha \alpha ;i}=-\frac{L^{2}}{\varepsilon ^{2}}A_{\alpha \alpha
}u^{h}u_{hi}-\frac{2}{L}u_{\alpha }u_{\alpha i}.  \label{lem2 temp1}
\end{equation}
By differentiating (\ref{lem2 temp1}), 
\begin{eqnarray}
A_{\alpha \alpha ;ii} &=&-\frac{L^{2}}{\varepsilon ^{2}}A_{\alpha \alpha
}u^{h}u_{hii}-\frac{L^{2}}{\varepsilon ^{2}}A_{\alpha \alpha }\sigma
^{ht}u_{ti}u_{hi}  \label{lem2 temp2} \\
&&-\frac{L^{2}}{\varepsilon ^{2}}A_{\alpha \alpha ;i}u^{h}u_{hi}-2\frac{L^{4}
}{\varepsilon ^{4}}A_{\alpha \alpha }u^{t}u_{ti}u^{h}u_{hi}  \notag \\
&&-\frac{2}{L}u_{\alpha }u_{i\alpha i}-\frac{2}{L}u_{i\alpha }^{2}+\frac{2L}{
\varepsilon ^{2}}u^{h}u_{hi}u_{\alpha }u_{i\alpha }.  \notag
\end{eqnarray}
Plugging (\ref{lem2 temp1}) into (\ref{lem2 temp2}), we infer that
\begin{eqnarray*}
A_{\alpha \alpha ;ii} &=&-\frac{L^{2}}{\varepsilon ^{2}}A_{\alpha \alpha
}u^{h}u_{iih}-\frac{L^{2}}{\varepsilon ^{2}}A_{\alpha \alpha }\sigma
^{ht}u_{ti}u_{hi} \\
&&-\frac{L^{4}}{\varepsilon ^{4}}A_{\alpha \alpha }u^{t}u_{ti}u^{h}u_{hi}-
\frac{2}{L}u_{\alpha }u_{ii\alpha } \\
&&-\frac{2}{L}u_{i\alpha }^{2}+\frac{4L}{\varepsilon ^{2}}
u^{h}u_{hi}u_{\alpha }u_{i\alpha }.
\end{eqnarray*}
By (\ref{sec3 temp0.3}),
\begin{equation*}
-\frac{L^{2}}{\varepsilon ^{2}}A_{\alpha \alpha }\sigma ^{ht}u_{ti}u_{hi}=-
\frac{L}{\varepsilon ^{2}}A_{\alpha \alpha }u_{ii}^{2}+\frac{L^{4}}{
\varepsilon ^{4}}A_{\alpha \alpha }u_{ii}^{2}u^{i}u^{i}.
\end{equation*}
Hence
\begin{eqnarray*}
A_{\alpha \alpha ;ii} &=&-\frac{L^{2}}{\varepsilon ^{2}}A_{\alpha \alpha
}u^{h}u_{iih}-\frac{L}{\varepsilon ^{2}}A_{\alpha \alpha }u_{ii}^{2} \\
&&-\frac{2}{L}u_{\alpha }u_{ii\alpha }-\frac{2}{L}u_{i\alpha }^{2}+\frac{4L}{
\varepsilon ^{2}}u_{ii}u_{i\alpha }u_{\alpha }u^{i} \\
&=&\frac{L^{2}}{\varepsilon ^{2}}A_{\alpha \alpha }u^{h}w_{ii;h}-\frac{L^{2}
}{\varepsilon ^{2}}A_{\alpha \alpha }u^{h}A_{ii;h}-\frac{L}{\varepsilon ^{2}}
A_{\alpha \alpha }u_{ii}^{2} \\
&&+\frac{2}{L}u_{\alpha }w_{ii;\alpha }-\frac{2}{L}u_{\alpha }A_{ii;\alpha }-
\frac{2}{L}u_{i\alpha }^{2}+\frac{4L}{\varepsilon ^{2}}u_{ii}u_{i\alpha
}u_{\alpha }u^{i}.
\end{eqnarray*}
Employing (\ref{lem2 temp1}) again, it follows
\begin{eqnarray}
A_{\alpha \alpha ;ii} &=&\frac{L^{2}}{\varepsilon ^{2}}A_{\alpha \alpha
}w_{ii;h}u^{h}+\frac{2}{L}u_{\alpha }w_{ii;\alpha }  \label{lem2 temp3} \\
&&-\frac{L}{\varepsilon ^{2}}A_{\alpha \alpha }u_{ii}^{2}-\frac{2}{L}
u_{i\alpha }^{2}+\frac{4L}{\varepsilon ^{2}}u_{ii}u_{i\alpha }u_{\alpha
}u^{i}  \notag \\
&&+\frac{L^{4}}{\varepsilon ^{4}}A_{\alpha \alpha }A_{ii}u_{ht}u^{h}u^{t}+2
\frac{L}{\varepsilon ^{2}}A_{\alpha \alpha }u_{ii}u_{i}u^{i}  \notag \\
&&+2\frac{L}{\varepsilon ^{2}}A_{ii}u_{\alpha \alpha }u_{\alpha }u^{\alpha }+
\frac{4}{L^{2}}u_{i\alpha }u_{\alpha }u_{i}.  \notag
\end{eqnarray}
Hence
\begin{eqnarray}
\hskip20pt{\sum }_{i,\alpha }w^{ii}A_{\alpha \alpha ;ii}w_{\alpha \alpha }
&=&\frac{L^{2}}{\varepsilon ^{2}}W\sum \varphi _{h}u^{h}+2\frac{L^{2}}{
\varepsilon ^{2}}\sum w_{ii}u^{i}\varphi _{i}  \label{lem2 temp3a} \\
&&-\frac{L}{\varepsilon ^{2}}W\sum w^{ii}u_{ii}^{2}-\frac{2}{L}\sum
u_{ii}^{2}  \notag \\
&&+\frac{L}{\varepsilon ^{2}}\sum \bigg\{
4u_{ii}^{2}u_{i}u^{i}+4u_{ii}u_{i}u^{i}+2Ww^{ii}u_{ii}u_{i}u^{i}\bigg\} 
\notag \\
&&+\frac{L}{\varepsilon ^{2}}\bigg\{W\mathcal{W}\sum u_{ii}u_{i}u^{i}+2
\mathcal{W}\sum w_{ii}u_{ii}u_{i}u^{i}\bigg\}.  \notag
\end{eqnarray}
Recalling (\ref{new1}), 
\begin{equation*}
\sum
\{4u_{ii}^{2}u_{i}u^{i}+4u_{ii}u_{i}u^{i}+2Ww^{ii}u_{ii}u_{i}u^{i}\}\leq Q,
\end{equation*}
and 
\begin{equation*}
W\mathcal{W}\sum u_{ii}u_{i}u^{i}+2\mathcal{W}\sum
w_{ii}u_{ii}u_{i}u^{i}\leq -\bigg\{W\mathcal{W}\sum w_{ii}u_{i}u^{i}+2
\mathcal{W}\sum w_{ii}^{2}u_{i}u^{i}\bigg\}+Q,
\end{equation*}
the second inequality of Lemma \ref{DDA} follows from (\ref{lem2 temp3a}).

From (\ref{lem2 temp3}) it follows that 
\begin{eqnarray*}
A_{\alpha \alpha ;ii}-A_{ii;\alpha \alpha } &=&\frac{L^{2}}{\varepsilon ^{2}}
A_{\alpha \alpha }w_{ii;h}u^{h}+\frac{2}{L}u_{\alpha }w_{ii;\alpha } \\
&&-\frac{L^{2}}{\varepsilon ^{2}}A_{ii}w_{\alpha \alpha ;h}u^{h}-\frac{2}{L}
u_{i}w_{\alpha \alpha ;i} \\
&&-\frac{L}{\varepsilon ^{2}}A_{\alpha \alpha }u_{ii}^{2}+\frac{4L}{
\varepsilon ^{2}}u_{ii}u_{i\alpha }u_{\alpha }u^{i} \\
&&+\frac{L}{\varepsilon ^{2}}A_{ii}u_{\alpha \alpha }^{2}-\frac{4L}{
\varepsilon ^{2}}u_{\alpha \alpha }u_{i\alpha }u_{i}u^{\alpha }.
\end{eqnarray*}
Hence 
\begin{eqnarray}
\sum w^{ii}\left( A_{\alpha \alpha ;ii}-A_{ii;\alpha \alpha }\right) &=&
\frac{\left( n+2\right) L^{2}}{\varepsilon ^{2}}\sum \varphi _{i}u^{i}
\label{lem2 temp3b} \\
&&-\frac{L^{2}}{\varepsilon ^{2}}\mathcal{W}\sum w_{\alpha \alpha ;h}u^{h}-
\frac{2}{L}\sum w^{ii}u_{i}w_{\alpha \alpha ;i}  \notag \\
&&-\frac{nL}{\varepsilon ^{2}}\sum w^{ii}u_{ii}^{2}+\frac{L}{\varepsilon ^{2}
}\mathcal{W}\sum u_{ii}^{2}.  \notag
\end{eqnarray}
Since 
\begin{equation*}
-\frac{nL}{\varepsilon ^{2}}\sum w^{ii}u_{ii}^{2}\geq -\frac{L}{\varepsilon
^{2}}Q
\end{equation*}
and 
\begin{equation*}
\frac{L}{\varepsilon ^{2}}\mathcal{W}\sum u_{ii}^{2}\geq \frac{L}{
\varepsilon ^{2}}\mathcal{W}\sum w_{ii}^{2}-\frac{L}{\varepsilon ^{2}}Q,
\end{equation*}
the first inequality of Lemma \ref{DDA} follows from (\ref{lem2 temp3b}).
\end{proof}

In view of Lemma \ref{DDA}, (\ref{sec3 temp2.4}) can be rewritten in the
form 
\begin{eqnarray}
\sum w^{ij}W_{ij} &\geq &\frac{L}{\varepsilon ^{2}}\mathcal{W}\sum
w_{ii}^{2}+\frac{L}{\varepsilon ^{2}}\mathcal{W}(W\sum
w_{ii}u_{i}u^{i}+2\sum w_{ii}^{2}u_{i}u^{i})  \label{sec3 temp3.1} \\
&&-\frac{L^{2}}{\varepsilon ^{2}}(\mathcal{W}+6)\sum w_{\alpha \alpha
;i}u^{i}+\Re _{\varphi }-\frac{L}{\varepsilon ^{2}}Q,  \notag
\end{eqnarray}
where 
\begin{equation}
\Re _{\varphi }=:-\frac{2L^{2}}{\varepsilon ^{2}}\sum w_{ii}u^{i}\varphi
_{i}+(n+6-W)\frac{L^{2}}{\varepsilon ^{2}}\sum \varphi _{\beta }u^{\beta
}+\sum \varphi _{\alpha \alpha }.  \label{sec3 def R}
\end{equation}

By (\ref{sec3 critical condition}), we have
\begin{equation}
\sum_{\alpha }w_{\alpha \alpha ;k}=\sum A_{\alpha \alpha ;k}w_{\alpha \alpha
}-W\frac{\eta _{k}}{\eta }.  \label{new2}
\end{equation}
It follows from (\ref{lem2 temp1}) that
\begin{eqnarray}
\sum_{\alpha }w_{\alpha \alpha ;k} &=&-W\frac{\eta _{k}}{\eta }-\frac{L^{2}}{
\varepsilon ^{2}}Wu_{kk}u^{k}-\frac{2}{L}w_{kk}u_{kk}u_{k}
\label{sec3 1st critical} \\
&=&-W\frac{\eta _{k}}{\eta }+\frac{L^{2}}{\varepsilon ^{2}}Ww_{kk}u^{k}-
\frac{L^{2}}{\varepsilon ^{2}}Wu^{k}  \notag \\
&&+\frac{2L^{2}}{\varepsilon ^{2}}w_{kk}^{2}u^{k}-\frac{2L^{2}}{\varepsilon
^{2}}w_{kk}u^{k}.  \notag
\end{eqnarray}
Hence, by \eqref{sec3 temp0.2} and \eqref{new1} 
\begin{eqnarray}
-\frac{L^{2}}{\varepsilon ^{2}}(\mathcal{W}+6)\sum w_{\alpha \alpha ;i}u^{i}
&\geq &-\frac{L}{\varepsilon ^{2}}Q+\frac{L^{2}}{\varepsilon ^{2}}W(\mathcal{
\ \ W}+6)\frac{\sum u^{i}\eta _{i}}{\eta }  \label{sec3 temp3.2} \\
&&-\frac{L}{\varepsilon ^{2}}\mathcal{W}(W\sum w_{ii}u^{i}u_{i}+2\sum
w_{ii}^{2}u_{i}u^{i}).  \notag
\end{eqnarray}
Therefore, by inserting (\ref{sec3 temp3.2}) into (\ref{sec3 temp3.1}), we
find that \eqref{sec3 temp1} can be written as 
\begin{eqnarray}
0 &\geq &\frac{L}{\varepsilon ^{2}}\mathcal{W}\sum w_{ii}^{2}+W\sum w^{ij}(
\frac{\eta _{ij}}{\eta }-2\frac{\eta _{i}\eta _{j}}{\eta ^{2}})
\label{sec3 temp3.4} \\
&&+\frac{L^{2}}{\varepsilon ^{2}}W(\mathcal{W}+6)\frac{\sum u^{i}\eta _{i}}{
\eta }+\Re _{\varphi }-\frac{L}{\varepsilon ^{2}}Q.  \notag
\end{eqnarray}
Without loss of generality, we may assume the cut-off function $\eta $
satisfies $|D\eta |^{2}\leq C\eta $ (otherwise we may replace $\eta $ by $
\eta ^{2}$) and $|D^{2}\eta |\leq C$. Hence it follows 
\begin{eqnarray*}
&&W\sum w^{ij}(\frac{\eta _{ij}}{\eta }-2\frac{\eta _{i}\eta _{j}}{\eta ^{2}}
)+\frac{L^{2}}{\varepsilon ^{2}}W(\mathcal{W}+6)\frac{\sum u^{i}\eta _{i}}{
\eta } \\
&\geq &-C\frac{W}{\eta }(\left\vert D^{2}\eta \right\vert +\frac{\left\vert
D\eta \right\vert ^{2}}{\eta })\sum w^{ij}\sigma _{ij}-C\frac{L^{2}}{
\varepsilon ^{2}}\frac{W}{\eta }(\mathcal{W}+1) \\
&\geq &-\frac{L}{\varepsilon ^{2}}Q,
\end{eqnarray*}
where (\ref{sec3 temp0.1}) is used in the last inequality. Therefore (\ref
{sec3 temp3.4}) can be written as 
\begin{equation}
0\geq \frac{L}{\varepsilon ^{2}}\mathcal{W}\sum w_{ii}^{2}+\Re _{\varphi }-
\frac{L}{\varepsilon ^{2}}Q.  \label{sec3 temp3.5}
\end{equation}

\begin{lemma}
\label{remaining term estimation} We have, at $x_{0}$, 
\begin{equation*}
\Re _{\varphi }\geq -\frac{L}{\varepsilon ^{2}}Q.
\end{equation*}
\end{lemma}

\begin{proof}
Recalling (\ref{sec2 v}), we have
\begin{equation}
v_{\alpha }=2u^{h}u_{h\alpha }.  \label{lem3 temp2.1}
\end{equation}
Differentiating (\ref{sec3 phi}) gives 
\begin{eqnarray}
\varphi _{\alpha } &=&-(n+2)\frac{L^{2}}{\varepsilon ^{2}}u^{h}u_{h\alpha }+
\frac{f_{\alpha }}{f}-\frac{\nabla _{\alpha }(g\circ T)}{g\circ T},
\label{lem3 temp1a} \\
\varphi _{\alpha \beta } &=&-2(n+2)\frac{L^{4}}{\varepsilon ^{4}}
u^{h}u^{t}u_{h\alpha }u_{t\beta }-\frac{(n+2)}{2}\frac{L^{2}}{\varepsilon
^{2}}v_{\alpha \beta }+\frac{f_{\alpha \beta }}{f}  \label{lem3 temp1b} \\
&&-\frac{f_{\alpha }f_{\beta }}{f^{2}}-\frac{\nabla _{\alpha \beta
}^{2}(g\circ T)}{g\circ T}+\frac{\nabla _{\alpha }(g\circ T)}{g\circ T}\frac{
\nabla _{\beta }(g\circ T)}{g\circ T}.  \notag
\end{eqnarray}
Inserting (\ref{lem3 temp1a}) and (\ref{lem3 temp1b}) into (\ref{sec3 def R}
), we obtain
\begin{eqnarray}
\Re _{\varphi } &\geq &2\frac{L^{2}}{\varepsilon ^{2}}\sum w_{ii}\frac{
\nabla _{i}(g\circ T)}{g\circ T}u^{i}-(n+6-W)\frac{L^{2}}{\varepsilon ^{2}}
\sum \frac{u^{\alpha }\nabla _{\alpha }(g\circ T)}{g\circ T}
\label{lem3 temp1} \\
&&-\sum \frac{\nabla _{\alpha \alpha }^{2}(g\circ T)}{g\circ T}-\frac{(n+2)}{
2}\frac{L^{2}}{\varepsilon ^{2}}\sum v_{\alpha \alpha }-\frac{L}{\varepsilon
^{2}}Q.  \notag
\end{eqnarray}
Differentiating (\ref{lem3 temp2.1}), we obtain
\begin{eqnarray*}
\sum v_{\alpha \alpha } &=&2\sum \sigma ^{th}u_{t\alpha }u_{h\alpha }+2\sum
u^{h}u_{\alpha \alpha h} \\
&=&2\sum \sigma ^{th}u_{t\alpha }u_{h\alpha }+2\sum u^{h}A_{\alpha \alpha
;h}-2\sum u^{h}w_{\alpha \alpha ;h}.
\end{eqnarray*}
By (\ref{sec3 temp0.2}), (\ref{sec3 temp0.3}) and \eqref{new1}, 
\begin{equation*}
\sum \sigma ^{th}u_{t\alpha }u_{h\alpha }=\frac{1}{L}\sum
(u_{ii}^{2}-u_{ii}^{2}u^{i}u_{i})\leq \frac{1}{L}Q.
\end{equation*}
From (\ref{lem2 temp1}), 
\begin{equation*}
\sum u^{i}A_{\alpha \alpha ;i}=-\frac{n+2}{L}u_{ii}u^{i}u_{i}\leq \frac{1}{L}
Q.
\end{equation*}
Also, by (\ref{sec3 1st critical}), 
\begin{eqnarray*}
-\sum u^{k}w_{\alpha \alpha ;k} &=&W\sum \frac{u^{k}\eta _{k}}{\eta }-\frac{1
}{L}(W-2)\sum w_{kk}u_{k}u^{k} \\
&&+\frac{1}{L}Wv-\frac{2}{L}\sum w_{kk}^{2}u_{k}u^{k} \\
&\leq &\frac{1}{L}Q.
\end{eqnarray*}
Therefore we have 
\begin{equation}
\sum v_{\alpha \alpha }\leq \frac{1}{L}Q.  \label{lem3 temp2.2}
\end{equation}
It then follows from (\ref{lem3 temp1})
\begin{eqnarray}
\Re _{\varphi } &\geq &\frac{2L^{2}}{\varepsilon ^{2}}\sum w_{ii}\frac{
\nabla _{i}(g\circ T)}{g\circ T}u^{i}  \label{lem3 temp3} \\
&&-(n+6-W)\frac{L^{2}}{\varepsilon ^{2}}\sum \frac{u^{i}\nabla _{i}(g\circ T)
}{g\circ T}  \notag \\
&&-\sum \frac{\nabla _{\alpha \alpha }^{2}(g\circ T)}{g\circ T}-\frac{L}{
\varepsilon ^{2}}Q.  \notag
\end{eqnarray}

Now we compute $\nabla _{\alpha }(g\circ T)$ and $\sum \nabla _{\alpha
\alpha }^{2}(g\circ T)$. By (\ref{sec2 DT}) we have 
\begin{equation}
\nabla _{\alpha }(g\circ T)=g_{k}T_{\alpha }^{k}=g_{\alpha }w_{\alpha \alpha
}.  \label{lem3 temp4.1}
\end{equation}
By differentiating \eqref{sec2 DT}, we have 
\begin{eqnarray*}
\sum \nabla _{\alpha \alpha }^{2}(g\circ T) &=&\sum g_{kl}T_{\alpha
}^{k}T_{\alpha }^{l}+\sum g_{k}\nabla _{\alpha }T_{\alpha }^{k} \\
&=&\sum g_{k}A^{kl}w_{l\alpha ;\alpha }+\sum g_{\alpha \alpha }w_{\alpha
\alpha }^{2}-\sum g_{k}A_{k\alpha ;\alpha }w_{\alpha \alpha }.
\end{eqnarray*}
Recalling that $A^{kl}=\delta _{kl}$ at $x_{0}$, we have 
\begin{equation*}
A^{kl}w_{l\alpha ;\alpha }=w_{k\alpha ;\alpha }=w_{\alpha \alpha
;k}+A_{k\alpha ;\alpha }-A_{\alpha \alpha ;k}.
\end{equation*}
By (\ref{new2}),
\begin{eqnarray*}
\sum \nabla _{\alpha \alpha }^{2}(g\circ T) &=&\sum g_{k}w_{\alpha \alpha
;k}+\sum g_{k}(A_{k\alpha ;\alpha }-A_{\alpha \alpha ;k}) \\
&&+\sum g_{\alpha \alpha }w_{\alpha \alpha }^{2}-\sum g_{k}A_{k\alpha
;\alpha }w_{\alpha \alpha } \\
&=&-W\sum g_{k}\frac{\eta _{k}}{\eta }+\sum g_{\alpha \alpha }w_{\alpha
\alpha }^{2}+\sum g_{k}(A_{k\alpha ;\alpha }-A_{\alpha \alpha ;k}) \\
&&+\sum g_{k}(A_{\alpha \alpha ;k}-A_{k\alpha ;\alpha })w_{\alpha \alpha }.
\end{eqnarray*}
Using the second formula in Lemma \ref{DA}, we get 
\begin{eqnarray}
\sum \nabla _{\alpha \alpha }^{2}(g\circ T) &=&-W\sum g_{k}\frac{\eta _{k}}{
\eta }+\sum g_{\alpha \alpha }w_{\alpha \alpha }^{2}+\frac{L^{2}}{
\varepsilon ^{2}}(W-n)\sum w_{kk}u^{k}g_{k}  \label{lem3 temp4.2} \\
&&+\frac{L^{2}}{\varepsilon ^{2}}(W-\sum w_{\alpha \alpha }^{2})\sum
u^{k}g_{k}.  \notag
\end{eqnarray}
Inserting (\ref{lem3 temp4.1}) and (\ref{lem3 temp4.2}) into (\ref{lem3
temp3}), we then obtain 
\begin{eqnarray}
(g\circ T)\Re _{\varphi } &\geq &W\sum g_{k}\frac{\eta _{k}}{\eta }+\frac{
2L^{2}}{\varepsilon ^{2}}\sum w_{kk}^{2}g_{k}u^{k}  \label{lem3 key} \\
&&-\frac{6L^{2}}{\varepsilon ^{2}}\sum w_{kk}u^{k}g_{k}-\sum g_{\alpha
\alpha }w_{\alpha \alpha }^{2}  \notag \\
&&-\frac{L^{2}}{\varepsilon ^{2}}(W-\sum w_{\alpha \alpha }^{2})\sum
u^{k}g_{k}-\frac{L}{\varepsilon ^{2}}Q.  \notag
\end{eqnarray}

By (\ref{sec3 temp0b}), 
\begin{eqnarray}
\sum g_{k}\frac{\eta _{k}}{\eta } &=&L\sum (\sigma ^{ij}+\frac{L^{2}}{
\varepsilon ^{2}}u^{i}u^{j})g_{i}\frac{\eta _{j}}{\eta }
\label{lem3 temp4.3.0} \\
&=&\frac{L}{\eta }\big(\sum \sigma ^{ij}g_{i}\eta _{j}+\frac{L^{2}}{
\varepsilon ^{2}}(\sum g_{i}u^{i})(\sum \eta _{j}u^{j})\big).  \notag
\end{eqnarray}
We have 
\begin{equation*}
\sum \sigma ^{ij}g_{i}\eta _{j}=\left\langle Dg,D\eta \right\rangle \leq C,
\end{equation*}
where $D$ is the normal derivative in $\mathbb{R}^{n}$ and $\left\langle
\cdot ,\cdot \right\rangle $ denotes the standard Euclidean metric.
Similarly, $\sum \sigma ^{ij}u_{i}g_{j}=\sum u^{j}g_{j}$, $\sum \sigma
^{ij}u_{i}\eta_{j}=\sum u^{j}\eta_{j}$ and $\sum \sigma ^{ij}g_{i}g_{j}$ are
all bounded by a universal constant $C$. Hence from (\ref{lem3 temp4.3.0}), 
\begin{equation}
\sum g_{k}\frac{\eta _{k}}{\eta }\geq -\frac{L}{\varepsilon ^{2}}\frac{C}{
\eta }.  \label{lem3 temp4.3}
\end{equation}
Employing (\ref{sec3 temp0.2}) and (\ref{new1}), 
\begin{equation}
(u^{k})^{2}=\frac{\varepsilon ^{2}}{L^{3}}u_{k}u^{k}\leq \frac{\varepsilon
^{2}}{L^{3}},  \label{lem3 temp4.4a}
\end{equation}
for any given $k$. Using (\ref{sec3 temp0.3}) then (\ref{sec3 temp0.2}), we
have 
\begin{equation*}
\sum \sigma ^{ij}g_{i}g_{j}=\frac{1}{L}\big(\sum g_{i}^{2}-\frac{L^{3}}{
\varepsilon ^{2}}(\sum u^{i}g_{i})^{2}\big).
\end{equation*}
It implies 
\begin{eqnarray}
\sum g_{i}^{2} &= &L\sum \sigma ^{ij}g_{i}g_{j}+\frac{L^{3}}{\varepsilon ^{2}
}(\sum u^{i}g_{i})^{2}  \label{lem3 temp4.4b} \\
&\leq & C\frac{L^{3}}{\varepsilon ^{2}} .  \notag
\end{eqnarray}
By (\ref{lem3 temp4.4a}) and (\ref{lem3 temp4.4b}) it follows that $
|g_{k}u^{k}|\le C$. Hence 
\begin{equation}
\sum w_{kk}^{2}g_{k}u^{k} \geq -CW^{2}  \label{lem3 temp4.4+}
\end{equation}
and 
\begin{equation}
\sum w_{kk}u^{k}g_{k}\leq CW.  \label{lem3 temp4.4}
\end{equation}
Moreover, in view of (\ref{sec3 temp0.1}) and (\ref{sec3 temp0.2}), 
\begin{equation*}
\sigma _{ii}=L+\frac{L^{3}}{\varepsilon ^{2}}u_{i}u^{i}\leq C\frac{L}{
\varepsilon ^{2}},
\end{equation*}
for any given $i$. Consequently, 
\begin{equation}
\sum g_{\alpha \alpha }w_{\alpha \alpha }^{2}\leq \left\vert
D^{2}g\right\vert \sum \sigma _{\alpha \alpha }w_{\alpha \alpha }^{2}\leq C 
\frac{L}{\varepsilon ^{2}}W^{2}.  \label{lem3 temp4.5}
\end{equation}
By virtue of (\ref{lem3 temp4.3}), (\ref{lem3 temp4.4+}), (\ref{lem3 temp4.4}
) and (\ref{lem3 temp4.5}), we obtain from (\ref{lem3 key}) that 
\begin{eqnarray*}
\Re _{\varphi } &\geq &-\frac{L}{\varepsilon ^{2}(g\circ T)}Q \\
&\geq &-\frac{L}{\varepsilon ^{2}}Q.
\end{eqnarray*}
This completes the proof.
\end{proof}

By Lemma \ref{remaining term estimation} and (\ref{sec3 temp3.5}), we get,
at $x_{0}$, 
\begin{eqnarray*}
0 &\geq &\frac{L}{\varepsilon ^{2}}\mathcal{W}\sum w_{ii}^{2}-C\frac{L}{
\varepsilon ^{2}}\Big(1+\frac{W}{\eta }+W^{2}+\frac{W}{\eta }\mathcal{W}\Big)
\\
&\geq &\frac{L}{\varepsilon ^{2}}\mathcal{W}\frac{W^{2}}{n}-C\frac{L}{
\varepsilon ^{2}}\Big(1+\frac{W}{\eta }+W^{2}+\frac{W}{\eta }\mathcal{W}
\Big) .
\end{eqnarray*}
Multiplying $n\eta ^{2}L$ to both sides of the above inequality, we obtain 
\begin{eqnarray}  \label{sec3 temp4}
0 &\geq &\frac{L^{2}}{\varepsilon ^{2}}\mathcal{W}(H^{2}-CH) -C\frac{L^{2}}{
\varepsilon ^{2}}(1+H^{2}) \\
&\geq & C\frac{L^{2}}{\varepsilon ^{2}}\mathcal{W}H^{2} -C\frac{L^{2}}{
\varepsilon ^{2}}(1+H^{2}).  \notag
\end{eqnarray}
Note that by (\ref{sec3 order for eigenvalues}), 
\begin{equation}
\mathcal{W} \geq \sum_{k\geq 2}\frac{1}{\lambda _{k}} \geq \Big(\prod_{k\geq
2} \frac{1}{\lambda _{k}}\Big)^{\frac{1}{n-1}} \geq C\lambda _{1}^{\frac{1}{
n-1} } \geq C\Big(\frac{W}{n}\Big)^{\frac{1}{n-1}},  \label{sec3 temp5}
\end{equation}
where $C$ is independent of $\varepsilon$. Hence from (\ref{sec3 temp4}) we
get 
\begin{equation}  \label{sec3 temp6}
0 \geq \frac{L^{2}}{\varepsilon ^{2}}H^{2+\frac{1}{n-1}} -C\frac{L^{2}}{
\varepsilon ^{2}}(1+H^{2}) .
\end{equation}
Hence $H\le C$ at $x_0$ and this completes the proof of Theorem \ref{main
estimate}.

\vskip30pt

\section{A counterexample to the Lipschitz regularity}

In the last section we proved that the eigenvalues of $DT_{\varepsilon }$
are uniformly bounded. In this section we give an example to show that the
$T_\varepsilon$ is not uniformly Lipschitz continuous for small $
\varepsilon>0$, i.e., the matrix $DT_{\varepsilon }$ is not uniformly
bounded, even though the densities $f$ and $g$ are smooth and positive, and
the domain $\Omega ^{\ast }$ is c-convex with respect to $\Omega$. 

Our counterexample will be obtained by finding a choice of $f$ and $g$ such that the monotonic optimal transport $T_0$ between them is not Lipschitz continuous. Even if we said that the convergence $T_\ve\to T_0$ is not straightforward, we can prove that a uniform Lipschitz bound on $T_\ve$ would imply such a convergence, and hence the same bound on $T_0$. Hence, if $T_0$ is not Lipschitz, then $T_\ve$ cannot be uniformly Lipschitz.

\begin{lemma}
Suppose that the sequence of transports $T_\ve$ is uniformly Lipschitz. Then the whole family $T_\ve$ converges uniformly as $\ve\to 0$ to the unique monotonic optimal transport for the cost $|x-y|$, which will be Lipschitz with the same Lipschitz constant.
\end{lemma}
\begin{proof}
By Ascoli-Arzel\`a's Theorem, the uniform Lipschitz bound implies the existence of a uniform limit up to subsequences. Obviously this limit map $T$ will be optimal for the limit problem, i.e. the Monge problem for cost $c(x,y)=|x-y|$ and will share the same Lipschitz constant as $T_\ve$.

We only need to prove that $T$ is monotonic along transport rays. Take $L_\ve(x)=\sqrt{\ve^2+|T_\ve(x)-x|^2}$: these maps are also uniformly Lipschitz and converge uniformly to $L(x)=|T_\ve(x)-x|$. Let us denote by $u_\ve$ the potentials for the approximated problems and by $u$ the potential for the limit problem. Due to the uniqueness of the Kantorovich potential $u$, since all the functions $u_\ve$ are $1-$Lipschitz, we have $u_\ve\to u$ uniformly. Moreover, $Du_\ve\deb Du$ and the convergence is actually strong (in $L^2$, for instance) if restricted to the set $Tu=\{|Du|=1\}$ (as a consequence of $|Du_\ve|\leq 1$, which implies that we also have $\int_{Tu} |Du_\ve|^2\to \int_{Tu} |Du|^2$: this turns weak convergence into strong, and hence also implies pointwise, convergence). 

The monotonicity of $T$ is proven if one proves $DL\cdot Du\leq 1$, since the direction of the transport rays is that of $-Du$. This inequality is needed on the set of interior points of transport rays, which are exactly points where $|Du|=1$. On these points we can use the weak convergence $DL_\ve\deb DL$ (weakly-* in $L^\infty$) and the strong convergence $Du_\ve\to Du$, which means that it is enough to get $DL_\ve\cdot Du_\ve\leq 1$, and then pass the inequality to the limit. This is the point where we use the uniform Lipschitz bound on $T_\ve$: without such a bound we could not have the suitable weak convergence of $DL_\ve$.

In order to estimate $DL_\ve$, we use \eqref{sec2 L invariant under coord} and \eqref{sec2 T}. We come back to the notation without the index $\ve$, and write $DL$, thus getting
$$DL\cdot Du=-D_iu \,(T^i_j-\delta_{ij})\,D_j u=L D_iu\, D^2_{ij}u\, D_ju+|Du|^2\, DL\cdot Du.$$
Then, we use \eqref{sec2 matrixA} and \eqref{sec2 matrixW} and the positivity of the matrix $w_{ij}$, to get
$$L\, D_iu \,D^2_{ij}u \,D_ju\leq |Du|^2\,(1-|Du|^2).$$
This implies 
$$(1-|Du|^2)\,DL\cdot Du\leq |Du|^2\,(1-|Du|^2),$$
which provides $DL\cdot Du\leq  |Du|^2\leq 1$ (notice that, for fixed $\ve>0$, the norm of the gradient $|Du|$ is strictly less than $1$, which allows to divide by $1-|Du|^2$).
\end{proof}

To construct the counterexample where $T_0$ is not Lipschitz, our idea is as follows. Let 
\begin{equation}  \label{sec4 eg1 segments}
\ell _{a}=\{(x,y)\ \text{in}\ \mathbb{R}^{2}\ |\ \mathbb{\ }y=\sqrt{a}\left(
x+2+a\right) ,x\in \left[ -2-a,1\right] \}
\end{equation}
be a family of line segments $\ell _{a}$, where $a\in \lbrack 0,1]$. It is
clear that the segments $\ell _{a}$ do not intersect with each other and $
\cup_{a\in \left[ 0,1\right] }\ell _{a}=\Delta _{ABC}$, where $\Delta _{ABC}$
denotes the triangle with vertices $A=(-3,0),B=(1,4)$ and $C=(1,0)$. Let 
\begin{eqnarray*}
f &\equiv& 1, \\
g &=& 1+\frac{1}{4}x+\eta \left( y\right)
\end{eqnarray*}
be two densities on $\Delta _{ABC}$. We first show that there exists a
smooth positive function $\eta $ such that $f, g$ satisfy the mass balance
condition 
\begin{equation}
\int_{\Delta _{P_{a}CQ_{a}}}f=\int_{\Delta _{P_{a}CQ_{a}}}g,\;\text{ for all 
 }a\in \lbrack 0,1].  \label{sec4 mass balance}
\end{equation}
Here $P_{a}=\left( -2-a,0\right) $ and $Q_{a}=\left( 1,\left( 3+a\right) 
\sqrt{a}\right) $ are the endpoints of $\ell _{a}$. We then prove that there
is a Lipschitz function $u$, which is the potential function to Monge's
problem in $\Delta _{ABC}$, with the densities $f, g$ given above. By (\ref
{sec4 mass balance}) we can construct a measure preserving mapping $T_0$,
which pushes the density $f$ to the density $g$, with $\{\ell_a\}$ as its
transfer rays. Using the potential $u$ and the duality we show that $T_0$ is
the optimal mapping of Monge's problem. By reflection in the $x$-axis, we
extend $T_0$ to the triangle $\Delta _{ABB^{\prime }}$, where $B^{\prime
}=\left( 1,-4\right)$ is the reflection of $B$. Then $T_0$ is not Lipschitz
at the interior point $(-2, 0)$.

\begin{lemma}
\label{densities}There exists a smooth positive function $\eta $, such that
( \ref{sec4 mass balance}) holds. This function satisfies $\eta(y)=O(y^2)$ as $y\to 0$.
\end{lemma}

\begin{proof}
By direct computation, 
\begin{eqnarray*}
\int_{\Delta _{P_{a}CQ_{a}}}f &=&\frac{1}{2}\sqrt{a}\left( 3+a\right) ^{2},
\\
\int_{\Delta _{P_{a}CQ_{a}}}g &=&\int_{-2-a}^{1}\int_{0}^{\sqrt{a}\left(
x+2+a\right) }\left( 1+\frac{1}{4}x +\eta \left( y\right) \right) dydx \\
&=&\frac{\sqrt{a}}{24}\left( 3+a\right) ^{2}\left( 12-a\right)
+\int_{-2-a}^{1}\int_{0}^{\sqrt{a}\left( x+2+a\right) }\eta \left( y\right)
dydx.
\end{eqnarray*}
In order that (\ref{sec4 mass balance}) holds, we need 
\begin{equation}
\frac{1}{24}a^{3/2}\left( 3+a\right) ^{2}=\int_{-2-a}^{1}\int_{0}^{\sqrt{a}
\left( x+2+a\right) }\eta \left( y\right) dydx.  \label{sec4 lem1 temp1}
\end{equation}

Differentiating (\ref{sec4 lem1 temp1}) with respect to $a$, we have 
\begin{equation*}
\frac{a}{24}\left( 9+7a\right) \left( 3+a\right) =\int_{-2-a}^{1}\left(
x+2+3a\right) \eta \left( \sqrt{a}\left( x+2+a\right) \right) dx
\end{equation*}
which is equivalent to 
\begin{equation}
\frac{a^{2}}{24}\left( 9+7a\right) \left( 3+a\right) =\int_{0}^{\left(
3+a\right) \sqrt{a}}\left( t+2a\sqrt{a}\right) \eta \left( t\right) dt.
\label{sec4 lem1 temp2}
\end{equation}
In order to find $\eta $ satisfying (\ref{sec4 lem1 temp1}) for all $a\in 
\left[ 0,1\right] $, we only need to solve (\ref{sec4 lem1 temp2}), since the equality in \eqref{sec4 lem1 temp1} is true for $a=0$.

Let us introduce 
\begin{equation}
y=\left( 3+a\right) \sqrt{a}.  \label{sec4 lem1 temp3}
\end{equation}
It is clear that $y$ is a strictly increasing function of $a$. Let $
a(y)=O(y^2)$ be the inverse function of \eqref{sec4 lem1 temp3}.
Differentiating (\ref{sec4 lem1 temp2}) in $y$ and using $a_{y}=\frac{2\sqrt{
a}}{3\left( a+1\right) }$, we obtain 
\begin{equation*}
\frac{\sqrt{a}}{36}\left( 27+45a+14a^{2}\right) =\frac{3\left( 1+a\right)
^{2}}{2\sqrt{a}}\eta \left( y\right) +\int_{0}^{y}\eta \left( t\right) dt,
\end{equation*}
Taking derivative again, we obtain 
\begin{equation}
\eta ^{\prime }\left( y\right) +\frac{q\left( a\left( y\right) \right) }{y}
\eta \left( y\right) =yp\left( a\left( y\right) \right) ,
\label{sec4 lem1 temp4}
\end{equation}
where 
\begin{eqnarray*}
q\left( a\right) &=&\frac{\left( 5a-1\right) \left( 3+a\right) }{3\left(
1+a\right) ^{2}}, \\
p\left( a\right) &=&\frac{27+135a+70a^{2}}{162\left( 1+a\right) ^{3}\left(
3+a\right) }.
\end{eqnarray*}
Solving (\ref{sec4 lem1 temp4}), one finds an explicit formula for $\eta$:
\begin{equation}
\eta \left( y\right) =\int_{0}^{y}t\,p\left( a\left( t\right) \right) \exp
\left( -\int_{t}^{y}\frac{q\left( a\left( \tau \right) \right) }{\tau }d\tau
\right) dt.  \label{sec4 lem1 temp5}
\end{equation}
It is clear that 
\begin{equation*}
-1\leq q\left( a\left( y\right) \right) \leq 0\text{ if }|y| <<1.
\end{equation*}
Hence 
\begin{eqnarray*}
0\leq \eta \left( y\right) &\leq &C\int_{0}^{y}t\exp \left( \int_{t}^{y}\frac{1}{
\tau }d\tau \right) dt \\
&\leq &Cy^{2}.
\end{eqnarray*}

From (\ref{sec4 lem1 temp5}) it follows that 
\begin{eqnarray}
  \eta \left( y\right) 
  &=&\int_{0}^{y}t\, p\left( a\left( t\right) \right) \exp
\left( -\frac{1}{2}\int_{a\left( t\right) }^{a\left( y\right) }\frac{5a-1}{
a\left( 1+a\right) }da\right) dt  \label{sec4 lem1 temp6} \\
&=&\frac{\sqrt{a}}{324\left( a+1\right) ^{3}}\int_{0}^{a}\frac{3\left(
s+1\right) \left( 27+135s+70s^{2}\right) }{\sqrt{s}}ds  \notag \\
&=&\frac{a\left( 10a^{3}+41a^{2}+54a+27\right) }{54\left( a+1\right) ^{3}}. 
\notag
\end{eqnarray}
In the last two equalities, $a$ is the function of $y$ determined by (\ref
{sec4 lem1 temp3}). Therefore $\eta $ is positive and smooth and satisfies the required conditions.
\end{proof}

\begin{remark}
\label{extension}From (\ref{sec4 lem1 temp3}), we can explicitly write 
\begin{equation*}
a\left( y\right) =h\left( y\right) +\frac{1}{h\left( y\right) }-2,
\end{equation*}
where 
\begin{equation*}
h\left( y\right) =\sqrt[3]{\sqrt{\frac{1}{4}y^{4}+y^{2}}+\frac{1}{2}y^{2}+1} 
\text{.}
\end{equation*}
It is clear that $a\left( y\right) $ is a smooth even function.
\end{remark}

\begin{lemma}
\label{potential}There exists a function $u:\Delta _{ABC}\rightarrow \mathbb{
\ \ R}$ satisfying 
\begin{equation*}
| u\left( p\right) -u\left( q\right) | \leq | p-q | ,\text{ }\forall \text{ }
p,q\in \Delta _{ABC},
\end{equation*}
and equality holds if and only if both $p$ and $q$ lie on a common segment $
\ell _{a}$.
\end{lemma}

\begin{proof}
We will construct a function $u:\Delta _{ABC}\rightarrow \mathbb{R}$, which
decreases linearly along all $\ell _{a}$.

For $\left( x,y\right) \in \Delta _{ABC}$, let $a=a\left( x,y\right) $ be
the solution of the equation 
\begin{equation}
y=\sqrt{a}\left( a+2+x\right) \text{.}  \label{sec4 tempa1}
\end{equation}
Hence $\left( x,y\right) \in \ell _{a}$. Differentiating (\ref{sec4 tempa1})
with respect to $x$ and $y$ respectively, we get 
\begin{equation}
0=\frac{a_{x}}{2a}y+\sqrt{a}\left( a_{x}+1\right)  \label{sec4 tempa2}
\end{equation}
and 
\begin{equation}
1=\frac{a_{y}}{2a}y+\sqrt{a}a_{y}\text{.}  \label{sec4 tempa3}
\end{equation}
It follows by (\ref{sec4 tempa2}) and (\ref{sec4 tempa3}) that 
\begin{equation}
a_{y}+\frac{a_{x}}{\sqrt{a}}=0  \label{sec4 tempa4}
\end{equation}
provided $a\left( x,y\right) \not=0$.

On the other hand, for $\left( x,y\right) \in \Delta _{ABC}$, the direction
vector of the line segment $\ell _{a}$ passing through $\left( x,y\right) $ is
given by 
\begin{equation}
\nu \left( x,y\right) =\left( \nu _{1},\nu _{2}\right) =-\frac{\left( 1, 
\sqrt{a\left( x,y\right) }\right) }{\sqrt{1+a\left( x,y\right) }}\text{.}
\label{sec4 mu}
\end{equation}
Hence, by (\ref{sec4 tempa4}), 
\begin{equation}
\partial _{y}\nu _{1}-\partial _{x}\nu _{2}=\frac{1}{2\left( 1+a\right)
^{3/2}}\left( a_{y}+\frac{a_{x}}{\sqrt{a}}\right) =0\text{,}
\label{sec4 tempa5}
\end{equation}
provided $a\left( x,y\right) \not=0$.

Fix a point $P=\left( -2,1\right) $. Let 
\begin{eqnarray*}
\gamma \left( t\right) &=&\gamma _{X}\left( t\right) \\
&=&\left( t\left( x+2\right) -2,1-t\left( 1-y\right) \right) ,
\end{eqnarray*}
$t\in \left[ 0,1\right] $. Then $\gamma $ is the segment joining $P$ and $
X=\left( x,y\right) \in \Delta _{ABC}$. Set 
\begin{equation*}
u\left( x,y\right) =\left( x+2\right) \int_{0}^{1}\nu _{1}\left( \gamma
\left( t\right) \right) dt+\left( y-1\right) \int_{0}^{1}\nu _{2}\left(
\gamma \left( t\right) \right) dt\text{.}
\end{equation*}
We claim that $u$ satisfies 
\begin{equation}
Du\left( x,y\right) =\nu \left( x,y\right) \ \text{on all segments}\ \ell_a.
\label{sec4 lem2 key}
\end{equation}
Indeed, for any point $X_{0}=\left( x_{0},y_{0}\right) \in \Delta _{ABC}$
with $a\left( x_{0},y_{0}\right) \not=0$, by (\ref{sec4 tempa5}) we have 
\begin{eqnarray}
u_{x}\left( x_{0},y_{0}\right) &=&\int_{0}^{1}\nu _{1}\left( \gamma
_{0}\left( t\right) \right) dt+\left( x_{0}+2\right) \int_{0}^{1}t\partial
_{x}\nu _{1}\left( \gamma _{0}\left( t\right) \right) dt  \label{sec4 du} \\
&&+\left( y_{0}-1\right) \int_{0}^{1}t\partial _{x}\nu _{2}\left( \gamma
_{0}\left( t\right) \right) dt  \notag \\
&=&\int_{0}^{1}\nu _{1}\left( \gamma _{0}\left( t\right) \right)
dt+\int_{0}^{1}t\frac{d}{dt}\nu _{1}\left( \gamma _{0}\left( t\right)
\right) dt  \notag \\
&=&\int_{0}^{1}\frac{d}{dt}\left(t\nu _{1}\left( \gamma _{0}\left( t\right)
\right)\right) dt =
\nu _{1}\left( x_{0},y_{0}\right) ,  \notag
\end{eqnarray}
where $\gamma _{0}=\gamma _{X_{0}}$ and we used $\partial_x\nu_2=\partial_y\nu_1$. Similarly, we have 
\begin{equation*}
u_{y}\left( x_{0},y_{0}\right) =\nu _{2}\left( x_{0},y_{0}\right) .
\end{equation*}
Taking limit, we see that \eqref{sec4 lem2 key} also holds on the segment $
\ell_{a\ |\ a=0}$.

As $\nu$ is a unit vector, hence from \eqref{sec4 lem2 key} we have 
\begin{equation*}
| u\left( p\right) -u\left( q\right) | \leq | p-q | ,\text{ }\forall \text{ }
p,q\in \Delta _{ABC},
\end{equation*}
and equality holds if and only if both $p$ and $q$ lie on a common segment $
\ell _{a}$. This completes the proof.
\end{proof}

As in \cite{CFM, TW}, one can show by Lemma \ref{densities} that there is a
unique measure preserving map $T_{0}$ from $( f,\Delta _{ABC})$ to $(
g,\Delta _{ABC}) $ such that $T_{0}(p) $ and $p$ lie in a common $\ell _{a}$
for all $p\in \Delta _{ABC}$, and satisfies the monotonicity condition 
\begin{equation*}
(T_{0}(p) -T_{0}(q)) \cdot ( p-q ) \geq 0\ \ \forall\ p,q\in \ell _{a} .
\end{equation*}
With the help of and Lemma \ref{potential}, we prove that this $T_{0}$ is
indeed optimal. This fact is classical in optimal transport theory, but we show it in details for the sake of completeness.

\begin{lemma}
\label{optimal}$T_{0}$ is an optimal mapping in the Monge mass
transportation problem from $\left( f,\Delta _{ABC}\right) $ to $\left(
g,\Delta _{ABC}\right) $.
\end{lemma}

\begin{proof}
Recall that the total cost functional is given by 
\begin{equation*}
\mathcal{C}\left( s\right) =\int_{\Delta _{ABC}}f\left( z\right) | z-s\left(
z\right) | dz,
\end{equation*}
where $s\in \mathcal{S}$, the set of measure preserving maps from $\left(
f,\Delta _{ABC}\right) $ to $\left( g,\Delta _{ABC}\right) $; and the
Kantorovich functional is defined as 
\begin{equation*}
I\left( \psi ,\varphi \right) =\int_{\Delta _{ABC}}f\psi +\int_{\Delta
_{ABC}}g\varphi ,
\end{equation*}
where $\left( \psi ,\varphi \right) $ are function pairs in the set 
\begin{equation*}
\mathcal{K}=\left\{ \psi \left( x\right) +\varphi \left( y\right) \leq | x-y
| \ \ \ \forall\ x, y\ \in \Delta_{ABC}\right\} .
\end{equation*}
For all $s\in \mathcal{S}$ and $\left( \psi ,\varphi \right) \in \mathcal{K}$
, we have
\begin{eqnarray}
I\left( \psi ,\varphi \right) &=&\int_{\Delta _{ABC}}f\left( z\right) \psi
\left( z\right) dz+\int_{\Delta _{ABC}}f\left( z\right) \varphi \left(
s\left( z\right) \right) dz  \label{sec4 temp} \\
&\leq &\int_{\Delta _{ABC}}f\left( z\right) | z-s\left( z\right) | dz  \notag
\\
&=&\mathcal{C}\left( s\right) .  \notag
\end{eqnarray}
That is 
\begin{equation*}
\sup_{\mathcal{K}}I\left( \psi ,\varphi \right) \leq \inf_{\mathcal{S}} 
\mathcal{C}\left( s\right) .
\end{equation*}
Let $u$ be the function constructed in the proof of Lemma \ref{potential},
and let $v=-u$. Then we have $\left( u,v\right) \in \mathcal{K}$. As $
T_{0}(p)$ and $p$ lie on the same line segment, Lemma \ref{potential}
implies 
\begin{equation*}
u\left( z\right) -u\left( T_{0}\left( z\right) \right) =| z-T_{0}\left(
z\right) | .
\end{equation*}
So the inequality in (\ref{sec4 temp}) becomes equality provided $\left(
\psi ,\varphi \right) =\left( u,v\right) $ and $s=T_{0}$. Therefore 
\begin{equation*}
\mathcal{C}\left( T_{0}\right) =I\left( u,v\right) \leq \sup_{\mathcal{K}
}I\left( \psi ,\varphi \right) \leq \inf_{\mathcal{S}}\mathcal{C}\left(
s\right) \text{.}
\end{equation*}
Hence $T_0$ is optimal and the segments $\ell _{a}$ are transfer rays of
Monge's problem.
\end{proof}

Let $B^{\prime }=\left( 1,-4\right) $ be the reflection of the point $B$ in
the $x$-axis and let $\Omega =\Omega ^{\ast }=\Delta _{ABB^{\prime }}$.
Extend the functions $f,g$ to $\Omega $ such that they are symmetric with
respect to the $x$-axis. From the proof of Lemma 5, one sees
that $f,g$ are smooth and satisfy the mass balance condition (\ref{sec2 mass
balance}). The fact that $\eta$ is quadratic close to $0$ shows that it can be reflected as a $C^{2}$ function, and Remark 1 shows that it is indeed smooth. It is also known \cite{MTW} that $\Omega $ and $\Omega ^{\ast }$
are $c$-convex with respect to each other (for the cost function $
c_{\varepsilon }$, $0\leq \varepsilon \leq 1$).

Also extend $T_0$ to $\Omega$ so that it is symmetric with
respect to the $x$-axis. By the uniqueness of monotone optimal mappings \cite
{FM}, $T_0$ is an optimal mapping of Monge's problem from $(f, \Omega)$ to $
(g, \Omega)$.

We claim that $T_{0}$ is not Lipschitz continuous at the point $q_{0}=(-2,0)$
. Let $D_{a,\delta }$ be the strip in $\Delta _{ABC}$ between the segments $
\ell _{a}$ and $\ell _{a+\delta }$, and let $q_{\sigma }=(-2,\sigma )$ be
the intersection of $\ell _{a}$ with the line $\{x=-2\}$, where $\delta
,\sigma >0$ are constants. Let $T_{0}(q_{\sigma })=(x_{\sigma },y_{\sigma })$
. As $T_{0}$ is measure preserving, we have (see the construction of the
optimal mappings in \cite{CFM, TW}) 
\begin{equation*}
\lim_{\delta \rightarrow 0}\;\frac 1\delta \int_{D_{a,\delta }\cap
\{x<-2\}}f(x,y)dxdy=\lim_{\delta \rightarrow 0}\;\frac 1\delta \int_{D_{a,\delta }\cap
\{x<x_{\sigma }\}}g(x,y)dxdy.
\end{equation*}
That is 
\begin{equation*}
\int_{-2-a}^{-2}(x+2+3a)dx=\int_{-2-a}^{x_{\sigma }}(x+2+3a)\Big(1+\frac{1}{
4 }x+\eta \left( \sqrt{a}(x+2+a)\right) \Big)dx.
\end{equation*}
Making the change $t=2+a+x$, we obtain 
\begin{equation*}
\int_{0}^{a}(t+2a)dt=\int_{0}^{x_{\sigma }+2+a}(t+2a)\Big(\frac{1}{2}+\frac{
t-a}{4}+\eta \left( \sqrt{a}\,t\right) \Big)dx
\end{equation*}
Since both $(t-a)$ and $\eta \left( \sqrt{a}\,t\right)$ tend to $0$ when $t,a\to 0$ (recall that $\eta (t)=O(t^{2})$), they are negligible in front of the constant $\frac 12$. This implies that, for small $a$, we should have 
\begin{equation}
x_{\sigma }\geq -2+(\sqrt{5}-2)a.\label{lab by me}
\end{equation}
Indeed, either $x_{\sigma }+2$ does not tend to $0$, in which case \eqref{lab by me} is satisfied, or it tends to $0$, in which case we can write, for small $a$,
\begin{equation*}
\int_{0}^{a}(t+2a)dt\leq \int_{0}^{x_{\sigma }+2+a}\frac 34(t+2a)dx.
\end{equation*}
Computing these integrals explicitly we get exactly the inequality \eqref{lab by me}.

On the other hand, by \eqref{sec4 eg1 segments}, we have $\sigma =a^{3/2}$.
Note that $x (0)=-2$. Hence 
\begin{equation}
\lim_{\sigma \rightarrow 0+}\frac{x(\sigma )-x(0)}{\sigma }\geq \frac{1}{4}
\lim_{\sigma \rightarrow 0+}a^{-1/2}=\infty .  \label{sec4 temp2}
\end{equation}
Our claim follows.

As $q_{0}=(-2,0)$ is an interior point of $\Delta _{ABB^{\prime }}$, we have
thus constructed positive, smooth densities $f,g$, and $c$-convex domains $
\Omega =\Omega ^{\ast }=\Delta _{ABB^{\prime }}$, such that the associated
optimal mapping $T_{0}$ is not Lipschitz at interior points.

As the triangle $\Delta_{ABB^{\prime }}$ is $c$-convex with respect to each
other, the optimal mapping $T_\varepsilon$ is smooth \cite{MTW}. By Lemma 4, one
has $T_{0}=\lim_{\varepsilon\to 0} T_\varepsilon$, and the above example
shows that $T_{\varepsilon }$ is not locally, uniformly Lipschitz continuous
as $\varepsilon\to 0$.

\section{Applications and perspectives}

The regularity problem for the Monge cost is very natural in transport theory and very difficult. For the moment, even the implication $f,g\in C^\infty\impl T_0\in C^0$ in a convex domain is completely open. The transport $T_0$, among all the optimal transports for the cost $|x-y|$ (for which there is no uniqueness), is likely to be the most regular and the easiest to approximate.

The present paper presented a strategy inspired by the previous results introduced in \cite{MTW} to get Lipschitz bounds, i.e. $L^\infty$ bounds on the Jacobian. Yet, it only allows for some partial bounds, and the counter-example of Section 4 shows that a Lipschitz result is not possible. 
However, in the same counter-example, the monotonic transport $T_0$ is a continuous map, and the point where a non-Lipschitz behavior is observed shows anyway the behavior of a $C^{0,\frac 23}$ map. Thus, it is still possible to hope for continuous, or even H\"older, regularity results on $T_0$.

We stress that these results could also be applied to the regularity of the transport density. The transport density is a notion which is specifically associated to the transport problem for the Monge cost (see \cite{FM}): it is a measure $\sigma$ which satisfies 
\begin{equation}\label{MK system}
\begin{cases}
\mathrm{div}\cdot (\sigma D u)=f-g&\mbox{ in }\Omega\\
|D u|\leq 1&\mbox{ in }\Omega,\\
|D u|= 1&\mbox{ a.e. on }\sigma>0,\end{cases}
\end{equation}
together with the Kantorovich potential $u$.

Several weak regularity results have been established, starting from the absolute continuity of $\sigma$ if either $f$ or $g$ are absolutely continuous, till the $L^p$ estimates $f,g\in L^p\impl \sigma\in L^p$ (see \cite{FM,Am,DePEvaPra,DePPra2,simpleproof}).

An explicit formula for $\sigma$ in terms of optimal transport plans or maps is available (we will not develop it here, see \cite{Am}) and most possible strategies for the regularity of the transport density need some continuity of the corresponding optimal transport. Yet, one of the advantages of working with $\sigma$ is that any optimal transport $T$ produces the same density $\sigma$. This allows for choosing the most regular one, for instance $T_0$, but requires anyway some regularity on it. Here is where our analysis comes into play (without, unfortunately, providing any exploitable result). But there are other features of the transport density that one could take advantage of: from the fact that it only depends on the difference $f-g$, one can decide to add any common density to both measures. For instance, if $f$ and $g$ are smooth densities with compact support on $\R^n$, it is always possible to add common background measure on a same convex domain $\Omega$ including both the supports. $\Omega$ can be chosen as smooth as we want, and we can for instance take $\Omega$ to be a ball. Also, one can add another common density to $f$ and $g$ so as to get $g=1$. This last trick allows to avoid some of the tedious computations of Section 3, since in this case $g(T)$ has not to be differentiated.

In any case, even with these simplifications, the continuity result is not  available for the moments. Possible perspectives of the current research involve the use of these partial estimates to prove continuity. 

Among the possible strategies
\begin{itemize}
\item Use the bounds on $DT_\ve$ to get estimates on the directions of the transport rays for the limit problem, and use them to estimate how much the disintegrations of $f$ and $g$ vary according to the rays. Using the fact that the monotonic optimal transport (in one dimension) continuously depends on the measures, one can hope for the continuity of $T_0$.
\item Use the fact that the bound on $W$ gives an $L^\infty$ bound on $\mathrm{div} (L Du)$ and, since $L$ depends on $|Du|$, one faces a highly non-linear and highly degenerate elliptic PDE where the goal would be to get uniform continuity results on $L Du$. This recalls what has been recently done in very degenerate elliptic PDEs for traffic applications (see \cite{SanVes,ColFig}), but seems (much) harder because $LDu$ is not a uniformly continuous function of $Du$. 
\item Write down some elliptic PDEs solved by some scalar quantities associated to $T_\ve$, for instance by $L$, and use the bounds on the matrices $A$ and $w$ that have been proven here in order to apply De Giorgi-Moser arguments (or their wider generalizations, see \cite{Dib} for a complete framework). Should it work, this would give H\"older continuity. Unfortunately, our attempts have not given any useful PDE so far.
\end{itemize}

All in all, up to the two-dimensional result of \cite{FGP} (which requires disjoint and convex supports), the search for continuous optimal transports for the original cost of Monge is still widely open.

\end{document}